\documentclass[12pt]{amsart}

\usepackage{amsgen,amsmath,amstext,amsbsy,amsopn,amsfonts,amssymb,dsfont}
\usepackage{amsthm}
\usepackage{centernot}
\usepackage[colorlinks=true, allcolors=blue]{hyperref}

\usepackage[alphabetic]{amsrefs}

\usepackage{rotating}
\newcommand{\Longnearrow}{%
        \begin{turn}{45}
              \raisebox{-1ex}{$\Longrightarrow$}
      \end{turn}
}

        \newcommand{\Longsearrow}{%
        \begin{turn}{-45}
                \raisebox{-1ex}{$\Longrightarrow$}
        \end{turn}
        }

        \newcommand{\Nlongnwarrow}{%
        \begin{turn}{135}
                \raisebox{-1ex}{$\centernot\Longrightarrow$}
        \end{turn}
        }

         \newcommand{\Nlongswarrow}{%
        \begin{turn}{-135}
                \raisebox{-1ex}{$\centernot\Longrightarrow$}
        \end{turn}
        }

        \newcommand{\Nlonguparrow}{%
        \begin{turn}{90}
                \raisebox{-1ex}{$\centernot\Longrightarrow$}
        \end{turn}
        }
         \newcommand{\Nlongdownarrow}{%
        \begin{turn}{-90}
                \raisebox{-1ex}{$\centernot\Longrightarrow$}
        \end{turn}
        }

\newtheorem{theorem}{Theorem}[section]
\newtheorem{lemma}[theorem]{Lemma}
\newtheorem{corollary}[theorem]{Corollary}
\newtheorem{proposition}[theorem]{Proposition}
 \newtheorem{defi}[theorem]{Definition}
\newenvironment{definition}{\begin{defi}\rm}{\end{defi}}
\newtheorem{exa}[theorem]{Example}
\newtheorem{exer}[theorem]{Exercice}

\newenvironment{example}{\begin{exa}\rm}{\end{exa}}
\newtheorem{rem}[theorem]{Remark}
\newenvironment{remark}{\begin{rem}\rm}{\end{rem}}
\newtheorem{rems}[theorem]{Remarks}

\renewcommand{\emph}[1]{{\bf #1}}
\newtheorem{thmx}{Theorem}
 
\newtheorem{corx}[thmx]{Corollary}

\def\H{\mathcal H}

\def\K{\mathcal K}
\def\P{\mathcal P}

\def\A{\mathcal A}
\def\B{\mathcal B}
\def\C{\mathcal C}

\def\O{\mathcal O}

\def\NN{{\mathbf N}}
\def\ZZ{{\mathbf Z}}
\def\CCC{{\mathbf C}}
\def\RRR{{\mathbf R}}
\def\QQ{\mathbf Q}
\def\RR+{{\mathbf R}^*}

\def\KK{\mathbf K}
\def\kk{\mathbf k}

\def\PP{{\mathbf P}}

\def\Q_p{{\mathbf Q}_p}
\def\S1{{\mathbf S}^1}

\def\la{\lambda}

\newcommand{\Ind}{\operatorname{Ind}}
\newcommand{\Sub}{\operatorname{Sub}}
\newcommand{\Stab}{\operatorname{Stab}}
\newcommand{\Rep}{\operatorname{Rep}}
\newcommand{\Omm}{\operatorname{Comm}}
\newcommand{\Prim}{\operatorname{Prim}}
\newcommand{\Un}{\mathds{1}}

\def\CRed{C^*_{\rm red}}

\newcommand{\act}{\curvearrowright}
\newcommand{\pr}{\operatorname{Prob}}
\def\ep{\varepsilon}

\allowdisplaybreaks

\author[B. Bekka]{Bachir Bekka}
\address{Bachir Bekka, Univ Rennes, CNRS, IRMAR--UMR 6625, Campus Beaulieu, F-35042 Rennes Cedex, France}
\email{bachir.bekka@univ-rennes1.fr}

\author[M. Kalantar]{Mehrdad Kalantar}
\address{Mehrdad Kalantar, Department of Mathematics, University of Houston, Houston, TX, USA}
\email{kalantar@math.uh.edu}

\begin{document}

\title[Quasi-regular representations]{Quasi-regular representations of discrete groups and associated $C^*$-algebras}

\thanks{BB was supported  by the  Agence Nationale de la Recherche (ANR-11-LABX-0020-01, ANR-14-CE25-0004); MK was supported by the NSF Grant DMS-1700259.}

\maketitle

\begin{abstract}
Let $G$ be a countable group. We introduce several equivalence relations on the set $\Sub(G)$ of subgroups of $G$, defined by properties of the quasi-regular representations $\la_{G/H}$ associated to $H\in \Sub(G)$ and compare them to the relation of $G$-conjugacy of subgroups.
We define a class  $\Sub_{\rm sg}(G)$  of subgroups (these are subgroups with a certain spectral gap property)
and show that they are rigid,  in the sense that
 the equivalence class of  $H\in \Sub_{\rm sg}(G)$ for any one of the above equivalence relations
 coincides with the $G$-conjugacy class of $H.$
Next, we  introduce a second class $\Sub_{\rm w-par}(G)$ of subgroups (these are subgroups which are
weakly parabolic in some sense) and we establish results concerning the ideal structure of the $C^*$-algebra $C^*_{\la_{G/H}}(G)$ generated by $\la_{G/H}$ for  subgroups $H$ which belong to either one of  the classes $\Sub_{\rm w-par}(G)$ and $\Sub_{\rm sg} (G)$.
Our results are valid,  more generally, for induced representations $\Ind_H^G \sigma$, where $\sigma$ is a representation
of   $H\in \Sub(G).$ 
\end{abstract}

\section{Introduction}
\label{S: Intro}
Let $G$ be a countable discrete group. By results of Glimm \cite{Glim--61} and Thoma \cite{Thom--68},
  the classification of unitary dual $\widehat{G}$, that is, the  set of irreducible unitary representations of $G$ up to unitary equivalence, is hopeless, unless $G$ is virtually abelian. By contrast, 
the  primitive ideal space $\Prim(G)$, that is, the set of irreducible unitary representations of $G$ up to weak equivalence (see Section~\ref{S: C*-algebras}), is  a more accessible dual space  of $G:$ indeed, $\Prim(G)$, equipped
with a natural Borel structure, is known to be  a standard Borel space for any countable group $G$ (see \cite{EffrosPrim}).

Examples of  irreducible unitary representations of $G$ are given by  the quasi-regular representations $(\la_{G/H},\ell^2(G/H))$  
of self-commensurating subgroups $H$ (see Subsection~\ref{S: IrredQuasiRegRep}).
Given two such subgroups $H$ and $L,$ a natural question is: when do $\la_{G/H}$ and $\la_{G/L}$ define the same point in 
$\Prim(G)$, that is, when are $\la_{G/H}$ and $\la_{G/L}$ weakly equivalent?
One of our main concerns in this paper is to study this question as well as related problems.

Let $\Sub(G)$  be the set of  subgroups of $G$ and 
$\Rep(G)$ the set of unitary representations of $G$ on a separable Hilbert space.
Consider the map 
$$
\Lambda: \Sub(G) \to \Rep(G), \qquad H\mapsto \la_{G/H}.
$$
The group $G$ acts by conjugation on $\Sub(G)$ and $\Lambda$ factorizes to a map
$$
\Sub(G)/{\sim_{\rm conj}} \to \Rep(G)/{\sim_{\rm un}},
$$
where $\Sub(G)/{\sim_{\rm conj}}$ is the set of conjugacy classes of subgroups of $G$ and $\Rep(G)/{\sim_{\rm un}}$ is the set of equivalence classes of unitary representations of $G$.

The sets $\Sub(G)$ and $\Rep(G)$ carry  natural topologies,  respectively the Chabauty topology and the Fell topology.
The map $\Lambda$ is continuous with respect to these topologies (see Proposition~\ref{Pro:FellChabautyTop} below) and so factorizes to a map 
$$
\Sub(G)/{\sim_{\rm w-conj}} \to \Rep(G)/{\sim_{\rm w-un}},
$$
where $\Sub(G)/{\sim_{\rm w-conj}}$ is the \textbf{quasi-orbit} space of $\Sub(G)$ with respect to the action of $G$  (see Section~\ref{S:RepEquivSubgroups}) and
$\Rep(G)/{\sim_{\rm w-un}}$ is the set of \textbf{weak equivalence} classes of unitary representations of $G.$

Unitary equivalence and weak equivalence of representations induce through the map $\Lambda: \Sub(G) \to \Rep(G)$
two other equivalence relations $\sim_{\rm rep}$ and $\sim_{\rm w-rep}$ on $\Sub(G):$
\begin{itemize}
\item $H\sim_{\rm rep} L$ if $\la_{G/H}$ and $\la_{G/L}$  are equivalent;
\item $H\sim_{\rm w-rep} L$ if $\la_{G/H}$ and $\la_{G/L}$  are weakly equivalent.
\end{itemize}
 So, we have four equivalence relations 
$\sim_{\rm conj}$, $\sim_{\rm w-conj}$, $\sim_{\rm rep}$ and $\sim_{\rm w-rep}$ on $\Sub(G),$
the strongest and the weakest of which being $\sim_{\rm conj}$ and $\sim_{\rm w-rep}$, respectively.
The following diagram summarizes the  relationships between these  equivalence relations (for more details, see Proposition~\ref{Pro-Def-RepEquivalentSub} below):
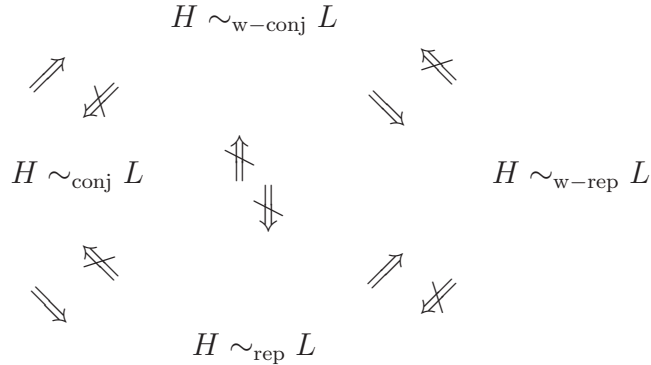
\bigskip
 $$
 \begin{array}[c]{ccccc}
&& H\sim_{\rm w-conj} L&&\\
&\Longnearrow\Nlongswarrow&&\Longsearrow\Nlongnwarrow&\\
&H\sim_{\rm conj} L&\Nlonguparrow\Nlongdownarrow&& H\sim_{\rm w-rep} L \\
&\Longsearrow \Nlongnwarrow&&\Longnearrow \Nlongswarrow&\\
&&H\sim_{\rm rep} L&&\\
\end{array}
$$

A natural question is: which subgroups $H$ of $G$ are  \textbf{rigid}  with respect to one of the relations
$\sim_{\rm w-conj}, \sim_{\rm rep},$ or $\sim_{\rm w-rep}$ in the sense that  the equivalence class of $H$ with respect to one of these relations coincides with  the conjugacy class of $H?$ In particular, which subgroups 
 have the strongest form of rigidity, that is, their $\sim_{\rm w-rep}$-equivalence
 class and the conjugacy class coincide?

In this view, we introduce the class $\Sub_{\rm sg}(G)$ of what we call 
subgroups with the \textbf{spectral gap property}; this is the class of   subgroups $H$ of $G$ such that the trivial representation $1_H$ is isolated in the spectrum of the natural representation of $H$ on $\ell^2(G/H).$
 The class $\Sub_{\rm sg}(G)$ contains all subgroups with Kazhdan's property (T) which are moreover strongly self-commensurating (see Definition~\ref{Def-StronglySelfComm} for this notion); the class $\Sub_{\rm sg}(G)$ 
 contains also all non-amenable \emph{a-normal} subgroups, that is, non-amenable subgroups $H$ of $G$ such that $H\cap gHg^{-1}$ is amenable for every  $g\in G\setminus H$. As an example, $H=SL_{n-1}(\ZZ)$ belongs to  $\Sub_{\rm sg}(G)$ for $G=SL_n(\ZZ)$ and $n\geq 3$ (see Example~\ref{Exa-SLnZSLmZ});  other examples include non-amenable maximal parabolic subgroups of convergence groups (see Proposition~\ref{Pro-ANormalParabolic} below).

Our first   result shows that subgroups from the class  $\Sub_{\rm sg}(G)$ are  rigid in the sense mentioned above.
\begin{thmx}
\label{MainTheorem1}
Let $G$ be a countable group and $H\in \Sub_{\rm sg}(G).$ 
Then   $H$  is representation rigid inside the class of self-commensurating subgroups
of $G:$   if $L\in \Sub(G)$ is  self-commensurating, then the representations  $\la_{G/H}$ 
and $\la_{G/L}$ are weakly equivalent if and only if $L$ is conjugate to $H.$ 
\end{thmx}

Our next result deals with the ideal  structure of the $C^*$-algebra  $C^*_{\lambda_{G/H}}(G)$,
that is, the closure in $\B(\ell^2(G/H))$ of the linear span of $\la_{G/H}(G)$ for the  norm topology
(see Subsection~\ref{SS-C*-W*-alg}).
For subgroups $H\in \Sub_{\rm sg}(G),$ we prove the following result.
\begin{thmx}
\label{MainTheorem2}
The $C^*$-algebra $C^*_{\lambda_{G/H}}(G)$ contains the ideal  $I$ of compact operators on  $\ell^2(G/H)$ 
and $I$ is  the  smallest non-zero ideal of $C^*_{\lambda_{G/H}}(G)$: every  non-zero 
 two-sided closed ideal  of $C^*_{\lambda_{G/H}}(G)$ contains $I.$
 \end{thmx}
 
Next, we introduce a second class of subgroups:    the class $\Sub_{\rm w-par}(G)$ of  what we call \emph{weakly parabolic} subgroups  of $G$; this is the class of subgroups $H$ of $G$  for which there exists a   topologically free boundary action $G\act X$  such that $X$ admits an $H$-invariant probability  measure (concerning the notion of a boundary action, see Subsection~\ref{SS-BoudaryActions} below).
When $G$ has a  topologically free boundary action $G\act X$, then every amenable subgroup
are well as every point stabilizer of $X$ belongs to the class $\Sub_{\rm w-par}(G)$; see 
also Example \ref{Exa-Sub_g}.

For the class of weakly parabolic subgroups, we prove the following result which generalizes
a criterion from \cite{KalKen}  for the reduced $C^*$-algebra $\CRed(G)$ to be simple. 

\begin{thmx}
\label{MainTheorem3}
Let $G$ be a countable group and $H\in \Sub_{\rm w-par}(G).$ 

 \begin{itemize}
  \item[{\normalfont (i)}] The regular representation $\lambda _G$ factorizes
to a representation  of  $C^*_{\lambda_{G/H}}(G)$; so,   the reduced  $C^*$-algebra $\CRed (G)$ of $G$
is isomorphic to $C^*_{\lambda_{G/H}}(G)/I$, where  $I$ is the kernel of $\lambda_G$ in $C^*_{\lambda_{G/H}}(G)$.
 \item[{\normalfont (ii)}]  $I$ is  the  largest proper ideal of $C^*_{\lambda_{G/H}}(G)$: every proper  
 two-sided closed ideal  of $C^*_{\lambda_{G/H}}(G)$ is contained in  $I.$
\end{itemize}
\end{thmx}
Theorem~\ref{MainTheorem2} has the following consequence for the space $\Prim(G)$.
\begin{corx} 
\label{MainCorollary}
Let $G$ be a countable group. The natural map 
$$\Sub_{\rm sg}(G)\to \Prim(G),  \qquad H \mapsto  [\la_{G/H}]$$
factorizes to an \textbf{injective} map  from the quotient space  $\Sub_{\rm sg}(G)/G$ to $\Prim(G),$
where $[\la_{G/H}]$ denotes the weak equivalence class of  $\la_{G/H}.$
\end{corx}
Actually, we prove much stronger  versions of Theorem~\ref{MainTheorem1}, Theorem~\ref{MainTheorem2}, Theorem~\ref{MainTheorem3},  and Corollary~\ref{MainCorollary}; indeed, our results are valid more generally for induced representations $\Ind_H^G \sigma$ for a representation $\sigma$ of $H$: see respectively Theorem~\ref{Theo-RepRigidSubgroups}, Theorem~\ref{Theo-IdealC*Subsg}, Theorem~\ref{Theo-C*IndRep}, 
and Theorem~\ref{Theo-RepRigidSubgroups-PrimIdeal} below.

We apply our results to the case of subgroups arising as point-stabilizers of actions $G\act X$ of $G$ on Hausdorff topological spaces $X$.  In particular (see Theorem \ref{thm:Furst-bnd-non-C*-simple}), we prove that, for the standard action 
$T\act \mathbf{S}^1$ of Thompson's group $T$ on the circle
and for $x,y\in  \mathbf{S}^1\setminus \{e^{2\pi i\theta} \mid \theta\in \QQ\}$ which 
do not belong to the same $T$-orbit, the following hold:
$$T_x \sim_{\rm w-conj} T_y \qquad\text{and} \qquad T_x \not\sim_{\rm rep} T_y,$$ 
where $T_x$ and $T_y$ denote the stabilizers of $x$ and $y$ in $T$, respectively.

Interesting examples of subgroups $H$ which belong to either
one of the classes $\Sub_{\rm sg}(G)$ or $\Sub_{\rm w-par}(G)$ arise as point-stabilizers of actions $G\act X$, where $X$ is extremally disconnected or where $G\act X$ is a topologically free boundary action (see Section \ref{S:Actions}  below).

One example of  a group  $G$ and a subgroup $H$ which belong to both classes $\Sub_{\rm sg}(G)$ and $\Sub_{\rm w-par}(G)$
is  given by $G=PSL_n(\ZZ)$   and $H=PSL_{n-1}(\ZZ)$   for $n\geq 3$ (see Example~\ref{Exa-SLnZSLmZ}).
Our results show  in particular that $\Prim(PSL_n(\ZZ))$ is infinite for $n\geq 3;$
 this last result is also true for $n=2$ (see \cite{Bekka-GLnQ}).

\vskip.2cm
This  paper is organized as follows. Besides this introduction, the paper contains six more sections. 
In Section 2, we gather some background material on representation theory of discrete groups that will needed 
throughout  the paper.
 We investigate in Section 3, connections between the equivalence relations introduced above and study rigidity properties of subgroups with respect to these equivalence relations.
 In Section 4, we prove Theorem~\ref{Theo-RepRigidSubgroups} about  rigidity of subgroups with the spectral gap property,
  which is a more general version of  Theorem~\ref{MainTheorem1}.
In Section 5, we study subgroups of  a group $G$ which stabilize a point  or a probability measure 
for a continuous action $G\act X$.
Section 6 is devoted to the study of weakly parabolic subgroups. 
In Section 7, we examine the ideal structure of $C^*$-algebras associated to induced  representations; we prove  there Theorem~\ref{Theo-IdealC*Subsg} and Theorem~\ref{Theo-C*IndRep}  which are more general versions of 
 Theorem~\ref{MainTheorem2} and  Theorem~\ref{MainTheorem3}.

\section{Preliminaries}
\subsection{The compact space of subgroups}
\label{SS-SubG}
Let $G$ be a countable discrete group. Let  $\Sub(G)$ be the set of subgroups of $G$, endowed with the Chabauty topology; this is the restriction to $\Sub(G)$ of the product topology on $\{0, 1\}^G$, where 
every subgroup $H\in \Sub(G)$ is identified with its characteristic function $\Un_H \in \{0, 1\}^G.$
 The space $\Sub(G)$ is compact and metrizable, and the group $G$ acts continuously on $\Sub(G)$ by conjugation:
$$
G \times \Sub(G)\to \Sub(G), \qquad (g, H)\mapsto gHg^{-1}.
$$
For $H\in \Sub(G),$ we denote by 
$\C(H)$ the $G$-orbit of $H$, that is, the conjugacy class of $H.$
This gives rise to a first equivalence relation ${\sim_{\rm conj}}$ on $\Sub(G)$ for which the equivalence
class of $H\in \Sub(G)$ is $\C(H).$
We will introduce further equivalence relations on $\Sub(G)$ in Section~\ref{S:RepEquivSubgroups}.

\subsection{$C^*$-algebras and von Neumann algebras associated to unitary representations}
\label{SS-C*-W*-alg}
Let $G$ be a countable discrete group. Every unitary representation $(\pi, \H)$ of $G$ extends uniquely to a representation (that is, a $*$-homomorphism) $\pi: \CCC[G]\to \B(\H)$ of the group  algebra $\CCC[G]$ of finitely supported functions on $G$ given by 
$$\pi(f)=\sum_{s\in G} f(s) \pi(s) \qquad\text{for} \quad f\in \CCC[G].$$ 
The $C^*$-algebra $C^*_{\pi}(G)$ generated by $\pi$ is the norm closure in $\B(\H)$ of the $*$-algebra $\{\pi(f)\mid f\in \CCC[G]\}.$ The von Neumann algebra $W^*_{\pi}(G)$ generated by $\pi$ is the closure in $\B(\H)$ of $C^*_{\pi}(G)$ for the strong (or weak) operator topology. 

Recall that the \textbf{maximal $C^*$-algebra} $C^*(G)$ of $G$ is the completion of 
 $\CCC[G]$  of $G$ with respect to the norm
$$
f\mapsto  \sup_{\pi\in \Rep(G)} \Vert \pi(f) \Vert,
$$
where $\Rep(G)$ is the set of unitary representations of $G$ on a separable Hilbert space.
Every $\pi\in \Rep (G)$ extends uniquely to a surjective $*$-homomorphism 
 $\pi: C^*(G)\to C^*_{\pi}(G)$.

A unitary representation $(\pi, \H)$ of $G$ is \textbf{factorial} if the 
von Neumann algebra $W^*_{\pi}(G)$   is  a factor;
in this case, $\pi$ is said to be \textbf{traceable } (or \textbf{normal}) if there exists a faithful normal (not necessarily finite) trace $\tau$ on $W^*_{\pi}(G)$  and a positive element $x\in C^*_{\pi}(G)$  such that $0<\tau(x)<+\infty.$
If $\pi$ is \textit{irreducible}, then $\pi$ is traceable 
if and only if the $C^*$-algebra $C^*_{\pi}(G)$ contains the algebra of compact operators on $\H$.
 For more details on all of this, see Chapters 6 and 17 in \cite{Dix--C*}.

For every subgroup $H$ of $G$, the injection map $i:H\to G$ extends to a $*$-homomorphism  $i_*:C^*(H)\to C^*(G)$ which is injective, and so $C^*(H)$ can be viewed as a subalgebra of $C^*(G)$.

We denote by $(\lambda_G, \ell^2(G))$ the (left) regular representation of $G$. The associated $C^*$-algebra 
$C^*_{\la_G}(G)$  is called the \emph{reduced $C^*$-algebra} of $G$ and will be denoted by $\CRed(G).$
The Dirac function $\delta_e$ at the neutral element $e\in G$ extends uniquely to a bounded trace $\tau_G$ on $W^*_{\lambda_G}(G)$,
called the \emph{canonical trace} on $G.$

\subsection{Weak containment and Fell's topology}
Let $G$ be a countable discrete group. Recall that  $\rho \in \Rep(G)$ is \textbf{weakly contained} in $\pi\in \Rep(G)$,
in symbols $\rho\prec \pi,$ if 
$$
\Vert \rho(x) \Vert \leq \Vert \pi(x) \Vert \qquad\text{for all} \qquad x\in C^*(G),
$$
or, equivalently, if $C^*\ker \pi\subset C^*\ker \rho$, where $C^*\ker \pi$ and $C^*\ker \rho$ denote the kernels of the extensions of $\pi$ and $\rho$ to $C^*(G)$.
The representations $\pi$ and $\rho$ are \textbf{weakly equivalent}, in symbols $\rho \sim \pi,$ 
if $\pi\prec \rho$ and $\rho\prec \pi$.

We equip $\Rep(G)$ with Fell's topology (see Appendix F in \cite{BHV} or \cite{Dix--C*}): a sequence $(\pi_n)_n$ in $\Rep(G)$ converges to $\pi\in \Rep(G)$ if every function of positive type associated to $\pi$ is the pointwise limit 
of sums of functions of positive type associated to $\pi_n$ as $n\to \infty$;
when $\pi$ is cyclic with cyclic vector $\xi,$ it suffices to check this for
the function of positive type  $g\mapsto\langle \pi(g)\xi\mid \xi\rangle$
associated to $\xi.$

Convergence in  $\Rep(G)$  can be expressed in terms of weak  containment:
$\lim_n \pi_n= \pi$ if and only if $\pi\prec \oplus_{k} \pi_{n_k}$ for every subsequence $(\pi_{n_k})_k$ of  $(\pi_n)_n$.

\subsection{States on $C^*$-algebras associated to unitary representations}
\label{SS-States }
We  recall a few facts about states of $C^*$-algebras (for more details, see Chap. 2 in  \cite{Dix--C*}).

Let  $\A$ be a unital  $C^*$-algebra, with unit element $\mathds{1}_\A$. A \textbf{state} of $\A$  is a positive linear functional $\varphi$ on $\A$  
with $\varphi(\mathds{1}_\A)=\Vert \varphi\Vert=1.$ The set of states of $\A,$ denoted $S(\A)$, is a convex and compact subset of 
the unit ball of the dual space $\A^*$, where $\A^*$ is endowed with the weak* topology.
Observe that in the case $\A=C(X)$ for a compact space $X,$ the state space
$S(\A)$ is  the space  $\pr(X)$ of probability measures on $X$.

Assume that $\B$ is a unital $C^*$- subalgebra of $\A$. Then every state
of $\B$ extends to a state of $\A$.
If $\mathcal{C}$ is a quotient $C^*$-algebra of $\A$,  then every state of $\mathcal{C}$ lifts to a state of $\A$
and so we can view $S(\mathcal{C})$ as weak* closed subset of $S(\A).$

Let $G$ be a discrete group and $\pi$ be a unitary representation 
of $G.$ There is natural (right) action  $G\act S(C^*_{\pi} (G))$  on the state space of $C^*_{\pi} (G)$ 
given  by
$$
 g\cdot \varphi(T)= \varphi (\pi(g) T \pi(g^{-1})),
$$
for $ g\in G,\   \varphi\in  S(C^*_{\pi} (G)),$ and  $T\in  C^*_{\pi} (G).$

Let $\pi, \rho$ be  unitary representations  of $G$. Viewing  $S(C^*_\pi(G))$ and  $S(C^*_\rho(G))$
as subsets of $S(C^*(G)),$ we have: 
$$\rho\prec \pi \Longleftrightarrow S(C^*_\rho(G)) \subset S(C^*_\pi(G)).$$
Assume, moreover,  that $\rho$ is cyclic; then $\rho\prec \pi$ if and only if 
the state $T\mapsto \langle \rho(T)\xi\mid \xi\rangle$ belongs to $S(C^*_\pi(G)),$ where
$\xi$ is a unit cyclic vector for $\rho.$
In particular, for  $\pi\in \Rep(G),$ we have 
$$\la_{G} \prec \pi \Longleftrightarrow \tau_G \in S(C^*_\pi(G)),$$ 
where $\tau_G\in S(C^*_\la(G)$ is the canonical trace on $G.$

\subsection{A  few facts about induced representations}
\label{S: FactsInducedRep}
Let $G$ be a countable discrete group. We collect a few known facts on \textbf{induced representations}  $\Ind_H^G \sigma$ for $H\in \Sub(G)$ and for a unitary representation $(\sigma, \K)$ of $H.$ Recall that  $ \Ind_H^G \sigma$ of $G$  may be realized as follows. 
 Let  $\H$ be the Hilbert space of maps $F : G \to \K$ with the following properties
\begin{itemize}
\item[(i)] 
$F(xh) = \sigma(h^{-1}) F(x)$ for all $x \in G, h \in H$;
\item[(ii)]
$\sum_{x \in G/H} \Vert F(x) \Vert^2 < \infty.$
\end{itemize}
 The induced representation $\pi = \Ind_H^G \sigma$ is given on $\H$ by left translation: 
$$
(\pi(g) F)(x) \, = \, F(g^{-1}x) 
\hskip.5cm \text{for all} \hskip.2cm 
g \in G, \hskip.1cm F \in \H
\hskip.2cm \text{and} \hskip.2cm 
x \in G.
$$

We give a decomposition of the restriction of $\Ind_H^G \sigma$ to another subgroup $L$ in terms of induced representations of $L$ and we recall how amenability of $H$ or $G$ can be expressed in terms of topological properties of $\Ind_H^G \sigma.$

For $H\in \Sub(G)$, $\sigma \in \Rep(H)$ and $g\in G,$ we denote by $\sigma^g$ the representation of $g^{-1} H g$ defined by  
$$\sigma^g(x)= \sigma (gxg^{-1}) \qquad\text{for all} \qquad x\in g^{-1} Hg.$$

The following lemma is well-known (for a proof, see for instance Proposition 9 in \cite{Bekka-GLnQ}).

 \begin{lemma}
\label{Lemma-Induced}
Let $G$ be a countable discrete group, $H,L \in \Sub(G)$ and $\sigma \in \Rep(H)$.
 Let  $S$ be a system of representatives for the double coset space $H\backslash G/L$.
 The restriction  of  $\Ind_H^G \sigma$  to $L$ is equivalent to the direct sum
$$\bigoplus_{s\in S} \Ind_{L \cap s^{-1}H s}^L (\sigma^s|_{L \cap s^{-1}H s}) .$$
\end{lemma}

Recall that, if $\pi$ is finite dimensional unitary representation of a group $G$, then 
$\pi\otimes \bar{\pi}$ contains the trivial representation $1_G$,
where $\bar\pi$ is the conjugate 
representation of $\pi$, and $\pi\otimes \rho$ denotes the (inner) tensor product 
of the  representations $\pi$ and $\rho$ (see \cite[Proposition A. 1.12]{BHV}).

Amenability of $G$ is characterized in terms of $\la_G$ by the Hulanicki-Reiter theorem (\cite[Theorem  G .3.2]{BHV}):
$$G \text{ is amenable} \iff 1_G\prec \la_G \iff  \pi\prec \la_G \text{ for every  } \pi\in \Rep(G) .$$
\begin{lemma} 
\label{Lem-FiniteDimInduced}
Let $\pi$ be a finite dimensional unitary representation of $G.$
The following properties are equivalent:
\begin{itemize}
\item[(i)] $\pi\prec \la_{G};$
\item[(ii)] $G$ is amenable.
\end{itemize}
\end{lemma}
\begin{proof}
The fact that (ii) implies (i) follows from the Hulanicki-Reiter theorem. Assume that 
$\pi\prec \la_{G}.$ Then 
$$ \pi\otimes\overline{\pi}\prec \la_G \otimes\overline{\la_G} \sim \la_G,$$
where we used the fact that  $\la_G\otimes \rho$ is equivalent to a multiple of $\la_G$ for 
every $\rho\in \Rep(G)$ (see e.g. \cite[Corollary  E 1.12]{BHV}).
Since $\pi\otimes \bar{\pi}$ contains the trivial representation $1_G$, it follows that $1_G\prec \la_G$ and so 
$G$ is amenable, by the Hulanicki-Reiter theorem.
\end{proof}

Here is a well-known consequence of Lemma~\ref{Lemma-Induced} and 
Lemma~\ref{Lem-FiniteDimInduced}.

\begin{corollary} 
\label{Cor-AmenableInduced}
For $H\in \Sub(G)$  and $\sigma\in \Rep(H),$
the following properties are equivalent:
\begin{itemize}
\item[(i)] $\Ind_{H}^G \sigma\prec \la_{G};$
\item[(ii)] $\sigma\prec \la_H.$
\end{itemize}
In particular, if $\sigma$ is finite dimensional, 
$\Ind_{H}^G \sigma\prec \la_{G}$ if and only if $H$ is amenable.
\end{corollary}
\begin{proof}
If $\sigma\prec \la_H$, then,  by ``continuity of induction"
(see \cite[Theorem F.3.5]{BHV}), 
$$\Ind_{H}^G \sigma\prec \Ind_{H}^G \la_H=\Ind_{H}^G (\Ind_{\{e\}}^H 1_H)= \la_G.$$

Conversely, assume that $\Ind_{H}^G \sigma \prec \la_{G}$. Then  $(\Ind_{H}^G \sigma)|_H\prec \la_{G}|_H$ and hence 
$(\Ind_{H}^G \sigma)|_H \prec \la_{H}$, since $\la_G|_H$ is a multiple of  $\la_H$
(see \cite[Remark F.1.9]{BHV}). By Lemma~\ref{Lemma-Induced}, 
$\sigma$ is a subrepresentation of $(\Ind_{H}^G \sigma)|_H$. It follows that $\sigma\prec \la_H$.

Assume now that $\sigma$ is finite dimensional. Then,  by Lemma~\ref{Lem-FiniteDimInduced}, $\sigma\prec \la_H$ if and only if $H$ is amenable.
\end{proof}

\subsection{Irreducible quasi-regular representations}
\label{S: IrredQuasiRegRep}

Let $G$ be a countable discrete group.
Recall that $H,L\in \Sub(G)$  are \textbf{commensurable} if $H\cap L$ has finite index in both $H$ and $L.$ The \textbf{commensurator} of $H$ in $G$ is the subgroup, denoted by $\Omm_G (H)$, of the elements $g \in G$ such that $ gH g^{-1}$ and $H$ are commensurable. The subgroup $H$ is \textbf{self-commensurating} if $\Omm_G (H)=H.$

For $H\in \Sub(G),$ denote by $\la_{G/H}$ the \textbf{quasi-regular} representation of $G$ on $\ell^2(G/H)$,
that is, $\la_{G/H}=\Ind_H^G 1_H.$ 

The subgroups $H$ for which $\la_{G/H}$ is irreducible  are described by the following classical result
(see \cite{Mack--51}, \cite{Kleppner}, \cite{Corwin}).

\begin{theorem}
\label{Theo-Mackey}
 \textbf{(Mackey's theorem)}
\begin{itemize}
 \item [(i)] Let $H$ be a self-commensurating subgroup of $G$ and let $\sigma$ a {\normalfont finite dimensional irreducible} unitary representation of $H$. Then $\Ind_{H}^G \sigma$ is irreducible.
  \item [(ii)]  Let   $H$ and $L$ be self-commensurating subgroups of $G,$ and let 
   $\sigma_1$ and $\sigma_2$ be finite dimensional irreducible unitary representations of $H$ and $L$, respectively. Then $\Ind_{H}^G \sigma_1$ and $\Ind_{L}^G \sigma_2$  are equivalent if and only if  there exists $g\in G$ such that
   \item $g^{-1}Lg=H$ and
   \item  $\sigma_1$ is equivalent to $\sigma_2^g$.
   \end{itemize}
\end{theorem}

Let $\widehat{G}$ be the unitary dual of $G,$ that is, the set of equivalence classes of \textit{irreducible}
unitary representations of $G$. 

Let $\Sub_{\rm sc} (G)$ be the subset of $G$ consisting of the self-commensurating subgroups of $G.$
Set 
$$
X:=\left\{(H, \sigma)\mid H\in \Sub_{\rm sc} (G), \sigma \in \widehat{H}_{\rm fd}\right\},
$$
where  $\widehat{H}_{fd}$ denotes the set of equivalence classes of {\it finite dimensional irreducible} unitary representations of $H$. 
There is a natural (left) action of $G$ on $X,$ given by 
$$
(g, (H, \sigma))\mapsto (gH g^{-1}, \sigma^{g^{-1}}).
$$
Denote by $X/G$ the space of $G$-orbits in $X.$
Identifying an irreducible representation  of $G$ with its equivalence class, Theorem~\ref{Theo-Mackey}
can be rephrased as follows:  
 \begin{theorem}
\label{Theo-MackeyRephrased}
The map 
$$
\Ind: X\to \widehat{G}, \qquad (H,\sigma) \mapsto  \Ind_H^G \sigma
$$
is well-defined and factorizes to an \emph{injective}  map 
$
X/G \to \widehat{G}.
$
\end{theorem}


\section{Representation equivalent subgroups}
\label{S:RepEquivSubgroups}

In this section, we introduce various equivalence relations on the space  of subgroups of a discrete group, investigate connections between them, and study rigidity properties with respect to these equivalence relations.

\subsection{Equivalence relations on the space of subgroups }
Let $G$ be a countable discrete group and let $\Sub(G)$ be the set of subgroups of $G,$ endowed with the Chabauty topology
as in Subsection~\ref{SS-SubG}.
\begin{definition} 
\label{Def-RepEquivalentSub}
Let $H, L  \in \Sub(G).$
\begin{itemize}
\item
$H$ and $L$ are  \textbf{weakly conjugate} (in symbols, $H \sim_{\rm w-conj} L$)  if  $\overline{\C(H)} =\overline{\C(L)}.$
\item
$H$ and $L$  are \textbf{representation equivalent} (in symbols, $H \sim_{\rm rep}  L$) 
if the associated quasi-regular representations $(\la_{G/H}, \ell^2(G/H)) $ and $(\la_{G/L},  \ell^2(G/L))$ are unitary equivalent. 
\item  
$H$ and $L$  are \textbf{weakly representation equivalent} (in symbols, $H \sim_{\rm w-rep}  L$)
if the associated quasi-regular representations $(\la_{G/H}, \ell^2(G/H)) $ and $(\la_{G/L},  \ell^2(G/L))$
are weakly equivalent.
\end{itemize}
\end{definition}
It is clear that all the relations introduced above are equivalence  relations on $\Sub(G)$.
We mention that  all the equivalence relations we have introduced
can be defined in exactly the same way on the space $\Sub(G)$ of {\it closed} subgroups of a general locally compact group $G$.

The   relation of weak conjugacy is related to the notion of  quasi-orbits: given an action  $G \curvearrowright X$ on a topological space $X$, consider the equivalence relation $\sim$ defined on $X$ by
$$
x \, \sim \, y
\hskip.5cm \text{if} \hskip.5cm
\overline{Gx} \, = \, \overline{Gy}.
$$
An equivalence class for $\sim$ is called a \textbf{quasi-orbit} for $G$.
The space $X/\sim,$ equipped with the quotient topology, is a T$_0$ topological space
and is characterized as such by a certain universal property; quasi-orbit spaces play an important r\^ole in the theory
of  $C^*$-algebras, among others.  For more on this, see \cite[Chap.6]{Williams}.

The notion of representation equivalent subgroups (of general locally compact groups) appeared  in the context of the  question of the existence of isospectral but not isometric Riemannian manifolds (see \cite{Sunada}, \cite{Gordon}, \cite{BhPR}).

We collect some facts about the  equivalence relations on $\Sub(G)$ introduced in  Definition~\ref{Def-RepEquivalentSub}.

\begin{proposition}
\label{Pro-Def-RepEquivalentSub} 
Let $H,L\in \Sub(G).$ Then 
\begin{itemize}
\item[{\normalfont (i)}] $H\sim_{\rm conj} L \implies H\sim_{\rm w-conj} L$;
 \item[{\normalfont (ii)}] $H\sim_{\rm conj} L \implies H\sim_{\rm rep} L$;
 \item[{\normalfont (iii)}]   $H\sim_{\rm w-conj} L \implies H\sim_{\rm w-rep} L;$
\item[{\normalfont (iv)}] $H\sim_{\rm rep} L \implies H\sim_{\rm w-rep} L;$ 
\item[{\normalfont (v)}] in general,  $H\sim_{\rm w-conj} L  \centernot\implies  H\sim_{\rm rep} L$;
\item[{\normalfont (vi)}] in general,  $H\sim_{\rm rep} L \centernot\implies  H\sim_{\rm conj} L;$
\item[{\normalfont (vii)}] in general,  $H\sim_{\rm rep} L \centernot\implies  H\sim_{\rm w-conj} L;$
\item[{\normalfont (viii)}] in general,  $H\sim_{\rm w-rep} L \centernot\implies  H\sim_{\rm rep} L.$
\end{itemize}
\end{proposition} 

\begin{proof}
(i) and (iv) are obvious.

\noindent
(ii) Assume that $H\sim_{\rm conj} L;$ we have to show that  $\la_{G/H}$ is equivalent to $\la_{G/L}$.
Indeed, if $L= gHg^{-1}$ for $g\in G,$ the map $U: \ell^2(G/H)\to  \ell^2(G/L),$
defined by 
$$Uf(x)= f(g^{-1}xg) \qquad\text{for} \qquad f\in \ell^2(G/H),\ x\in G,$$
is a bijective linear isometry.
Moreover, $U$  intertwines $\la_{G/H}$ and $\la_{G/L}^{g^{-1}}$.
 As  $\la_{G/L}^{g^{-1}}$ is unitary equivalent to 
$\la_{G/L},$ this proves the claim.

\noindent
(iii) For the proof, see Corollary~\ref{Cor1} below.

\noindent
(v) For the proof, see Theorem \ref{thm:Furst-bnd-non-C*-simple} below.
   
  \noindent
(vi) There are examples of finite groups $G$ containing non conjugate subgroups $H,L$
such $\la_{G/H}$ and $\la_{G/L}$ are equivalent (see \cite[I.8]{Berard}).

 \noindent
(vii)  Let $G, H$ and $L$ be  as in (vi); then $H\sim_{\rm rep} L$
and $H\centernot \sim_{\rm conj} L$ and hence $H\centernot \sim_{\rm w-conj} L$, since $G$ is finite.

   \noindent
(viii) Let $G=F_2$ be the free group on two 
   generators $a,b$ and let  $H$ be the subgroup generated by $a$ and $L=\{e\}$.
   On the one hand,  $\la_{G/H} $ is weakly contained in $ \la_{G}=\la_{G/L},$ since $H$ is amenable
   (see Corollary~\ref{Cor-AmenableInduced}).
   On the other hand, $\lim_{n} b^n H b^{-n}= \{e\}$ and hence  $\la_{G}$ is weakly contained in $\la_{G/H},$ 
   by Corollary~\ref{Cor1} below. Hence, we have $H\sim_{\rm w-rep} \{e\}$. 
   However, $\la_{G/H}$ is not equivalent to $ \la_{G}=\la_{G/L}.$ 
Indeed, $H$ has a non-zero invariant vector in $\ell^2(G/H)$ and no such vector in $\ell^2(G).$
Hence,   $H\centernot \sim_{\rm rep} \{e\}.$
\end{proof}

\subsection{Weak conjugacy and weak representation conjugacy}
\label{SS:WConj-RepConj}
We investigate the relationship between weak conjugacy and weak representation conjugacy of subgroups of a countable discrete group $G$.

The following result is a special case of a far more general result of Fell, valid for induced representations of locally compact groups (see \cite[Theorem 4.2]{Fell}).
\begin{proposition}
\label{Pro:FellChabautyTop}
The map 
$$
\Lambda: \Sub(G) \to \Rep(G), \qquad H\mapsto \la_{G/H}
$$
is continuous.
\end{proposition}

\begin{proof} 
Let $(H_n)_n$ be a sequence in  $\Sub(G)$ converging to $H\in \Sub(G).$
Then $\Un_{H_n}$ and $\Un_{H}$ are functions of positive type on $G$ associated to 
$\la_{G/H_n}$ and $\la_{G/H}$, respectively; moreover, we have
$$\lim_n \Un_{H_n}(g)= \Un_H(g) \qquad \text{for every} \qquad g\in G,$$ 
 by definition of the Chabauty topology on $\Sub(G).$
Since $\delta_H\in \ell^2(G/H)$ is a cyclic vector for $\la_{G/H}$
and since $\Un_H= \langle \la_{G/H} (\cdot) \delta_H \mid \delta_H\rangle,$ 
it follows that  $\lim_n \la_{G/H_n}= \la_{G/H}$ in the Fell topology.
\end{proof} 

\begin{corollary}
\label{Cor1}
Let $H, L\in \Sub(G).$ 
\begin{itemize}
\item[{\normalfont (i)}] Assume that $L\in \overline{\C(H)}.$
Then $\la_{G/L}\prec \la_{G/H}$.
\item[{\normalfont (ii)}]  $H\sim_{\rm w-conj} L \implies H\sim_{\rm w-rep} L.$
\end{itemize}
\end{corollary}

\begin{proof}
(i) Let $L\in \overline{\C(H)}$, and let $H_n\in \C(H)$ be such that $\lim_n H_n = L$. Then, by continuity of the map $\Lambda$, it follows that 
$$\lim_n \la_{G/{H_n}} = \lim_n \Lambda(H_n) = \Lambda(L) = \la_{G/L}.$$
 Since $\la_{G/{H'}}$ is equivalent to $\la_{G/{H}}$ for every $H'\in \C(H)$, we have therefore $\la_{G/L}\prec \la_{G/H}$.

\noindent
(ii) This follows immediately from Item (i) and the definitions.
\end{proof}

\begin{remark}
\label{Rem-Cor1}
\begin{itemize}
\item [(i)] The converse in Corollary~\ref{Cor1}.i does not hold in general:  for every amenable subgroup $H$ of $G,$ we have $\la_{G/H} \prec \la_{G}$ (see Corollary~\ref{Cor-AmenableInduced}); however, $H\in \overline{\{e\}}$  only if  $H=\{e\}.$
\item [(ii)] As shown in Proposition~\ref{Pro-Def-RepEquivalentSub}.vii, the converse in Corollary~\ref{Cor1}.ii does not hold in general. 
   \end{itemize}
\end{remark}

Let $\Sub_{\rm a}(G)$ be the subset of $\Sub(G)$ consisting of amenable subgroups of $G$. The following well-known fact
(see \cite[Corollary 4]{CapraceMonod}) is another consequence of Proposition \ref{Pro:FellChabautyTop}.

\begin{corollary}
\label{Cor-AmenableSubgroups}
Let $G$ be countable group. Then $\Sub_{\rm a}(G)$ is closed in $\Sub(G)$.
\end{corollary}
\begin{proof}
Let $(H_n)_n$ be a sequence in $\Sub_{\rm a}(G)$ converging to $H\in \Sub(G).$ 
Then $\lim_n \la_{G/H_n}= \la_{G/H},$ by Proposition~\ref{Pro:FellChabautyTop}.
Since $H_n$ is amenable, we have $\la_{G/H_n}\prec \la_G$ for every $n$.
Thus, we have $ \la_{G/H}\prec \la_G$, which implies that $H$ is amenable (see 
Corollary~\ref{Cor-AmenableInduced}).
\end{proof}

\subsection{Conjugation rigidity of subgroups}
\label{S: ConjRigid}
We define notions of rigidity for subgroups of a countable discrete group $G$, related to the various equivalence relations on $\Sub(G)$ introduced in Definition~\ref{Def-RepEquivalentSub}.

\begin{definition}
\label{Def-RepRigidSubgroups}
Let  $H\in\Sub(G).$
\begin{itemize}
\item[(i)] $H$ is said to be \textbf{conjugation rigid} 
 if, for  every  $L\in \Sub(G)$, we have 
$$L\sim_{\rm w-conj} H \implies L\sim_{\rm conj} H$$
\big(in other words, if $\overline{\C(L)}=\overline{\C(H)}\implies \C(L)=\C(H)$\big);

\item[(ii)] $H$ is said to be \textbf{representation rigid} 
if, for every  $L\in \Sub(G)$, we have 
$$L\sim_{\rm w-rep} H \implies L\sim_{\rm conj} H$$
\big(in other words, if $\la_{G/L}\sim \la_{G/H}\implies \C(L)=\C(H)$\big);

\item[(iii)]  $H$ is said to be \textbf{strongly representation rigid} if
the following holds: if $\Ind_L^G \sigma \sim \Ind_H^G \pi$ for $L\in \Sub(G), \sigma\in  \widehat{L}_{fd}$ and $\pi\in  \widehat{H}_{fd}$,
then  there exists $g\in G$ such that $L= g^{-1}Hg$ and $\sigma$ is equivalent to $\pi^{g}$.
\end{itemize}

Let $\H$ be a $G$-invariant subset of $\Sub(G)$ and $H\in \H.$ One may define conjugation rigidity, representation rigidity and strong representation rigidity of $H$ \textbf{inside $\H$}, by allowing only subgroups $L\in \H$  in (i), (ii), and (iii).
\end{definition}

\begin{remark}
\label{Rem-Def-RepRigidSubgroups}
 Let  $H\in\Sub(G).$
\begin{itemize}
\item[(i)] It follows from Corollary~\ref{Cor1} that if $H$  is representation rigid, then $H$ is conjugation rigid. 
\item[(ii)] $H$ may be conjugation rigid without being representation rigid; indeed, 
let $G=F_2$ be the free group on two generators $a,b$.
The trivial subgroup $H=\{e\}$ is obviously conjugation rigid; however, for $L=\langle a\rangle$,
 the representation $\la_{G/L}$ is weakly  equivalent to $\la_{G}$ (see proof of Proposition~\ref{Pro-Def-RepEquivalentSub}.viii),
 that is, $H\sim_{\rm w-rep} L$. So, $H$ is not representation rigid.
 \item[(iii)] Generalizing (ii), let  $G$ be  \textbf{$C^*$-simple} group, that is, a group $G$ for which the reduced $C^*$-algebra $\CRed(G)$  is simple; for large classes examples of such groups, see \cite{BKKO}). Then, for every amenable subgroup $L$ of $G$, we have $L\sim_{\rm w-rep} \{e\}$;  so, $H=\{e\}$ is conjugation rigid and not representation rigid if $G\neq \{e\}.$
\end{itemize}
\end{remark}

\section{A  class of rigid subgroups}
\label{S: RigidSubgroups}
We   exhibit a class of subgroups $H$ of $G$  which
 have rigidity properties  in the sense of Definition~\ref{Def-RepRigidSubgroups}.

Let $G$ be a countable discrete group acting on a set $X$. 
Let $\pi_X$ be the associated natural representation of $G$ on 
$\ell^2(X).$ We say that the action $G \curvearrowright X$  has a \textbf{spectral gap} if $\pi_X$ does not weakly contain 
the trivial representation $1_G$; this is the case  if and only if  there exists no $G$-invariant mean on $\ell^\infty(X).$  For more on group actions with a spectral gap, see the overview \cite{Bekka-SpectralGap}.

Let $H\in \Sub(G).$ Observe that   $H$ is a global fixed point for the natural action $H\curvearrowright  G/H$.

\begin{definition}
\label{Def-SpectSubgroups}
We say that a subgroup $H\in\Sub(G)$ has the \textbf{spectral gap property}  if 
the action $H\curvearrowright X$ has a spectral gap, where $X=  G/H\setminus \{H\}$.
We denote by $\Sub_{\rm sg}(G)$ the set of all  subgroups of  $G$ with the spectral gap property.
\end{definition}

\begin{remark}
\label{Rem-Def-SpectralSubgroups}
\begin{itemize}
 \item[(i)] Let $H\in\Sub_{\rm sg}(G)$. Then  $H$ is the only $H$-periodic point 
 in $G/H$
 (that is, the only point with finite 
$H$-orbit). Indeed, let $\O\subset  G/H$ be a finite $H$-orbit. Then $\mathds{1}_{\O}$ is an $H$-invariant non-zero function in $\ell^2(G/H)$ and hence $\O= \{H\}$.
\item[(ii)]  By Lemma~\ref{Lemma-Induced},  a subgroup $H\in \Sub(G)$ has the spectral gap property if and only if   $1_H$ is not weakly contained in  the direct sum
$$
\bigoplus_{s\in G \setminus H} \lambda_{H/ (H \cap s^{-1}H s)}.
$$
\item[(iii)]  Let $H\in\Sub_{\rm sg}(G)$. Then it follows from part (i) above that  $H$ is self-commensurating.
In fact, $H$ satisfies a slightly stronger property to be introduced in Definition~\ref{Def-StronglySelfComm}
below.
\end{itemize}
\end{remark}

 We now establish  a strong rigidity property of the class  $\Sub_{\rm sg}(G).$ 
 Recall that traceable representations have been defined in Subsection~\ref{SS-C*-W*-alg}.
 
 \begin{theorem}
\label{Theo-RepRigidSubgroups}
Let $G$ be a countable discrete group and $H\in \Sub_{\rm sg}(G)$.
\begin{itemize}
\item[{\normalfont (i)}] Let $\sigma \in \widehat{H}_{fd}$.
 Then $\Ind_H^G \sigma$ is a traceable irreducible representation of $G.$
 \item[{\normalfont (ii)}]  Let $\sigma \in \widehat{H}_{fd}$ and $\pi\in \widehat{G}$. If
 $\pi\sim \Ind_H^G \sigma$, then $\pi$ is unitary equivalent to $\Ind_H^G \sigma.$
 \item[{\normalfont (iii)}] The subgroup  $H$  is strongly representation rigid
inside the class  $\Sub_{\rm sc}(G)$ of self-commensurating subgroups:
if $\Ind_L^G \sigma \sim \Ind_H^G \pi$ for   $L\in\Sub_{\rm sc}(G),$
$ \sigma\in  \widehat{L}_{fd}$ and $\pi\in  \widehat{H}_{fd}$,
then  there exists $g\in G$ such that $H= g^{-1}Lg$ and $\pi$ is equivalent to $\sigma^{g}$.
\end{itemize}
\end{theorem}
\begin{proof}
By  Mackey's theorem  (Theorem~\ref{Theo-Mackey}), Item (iii)  is a special case of Item (ii).
Item (ii) follows from Item (i), in combination with a well-known  fact about irreducible representations of $C^*$-algebras containing the algebra of compact operators  (see Corollary 4.1.10 in \cite{Dix--C*}).
So, we only have to prove Item (i).

As a preliminary  step, we first establish that $H$ enjoys  a  spectral gap property
which is formally stronger than the one stated in Remark~\ref{Rem-Def-SpectralSubgroups}.ii.

\vskip.5cm
\noindent
$\bullet$ \textbf{Step 1.} For every $s\in G \setminus H,$ let 
$\sigma_s \in \Rep(H\cap s^{-1}H s).$ We claim that   the direct sum
$$
\rho:=\bigoplus_{s\in G \setminus H} \Ind_{H \cap s^{-1}H s}^H \sigma_s
$$
does not weakly contain any finite dimensional unitary representation of $H.$

Indeed, assume, by contradiction, that $\rho$
weakly contains a finite dimensional unitary representation of $H.$
Then, $\rho\otimes \bar{\rho}$ weakly contains $1_H.$ 
On the other hand, $\rho\otimes \bar{\rho}$ is equivalent to 
$$ \bigoplus_{s\in G \setminus H}  \Ind_{H \cap s^{-1}H s}^H \rho_s, $$
 where $\rho_s= \sigma_s \otimes (\bar{\rho}|_{H\cap s^{-1}H s})$;
see \cite[Proposition E. 2.5]{BHV}. So, denoting by $\K_s$ the Hilbert space of 
$\rho_s$ for $s\in G\setminus H,$ there exists  a sequence $(F_n^s)_{n\geq 1}$ of maps
$F_n^s : H\to \K_s$ in the Hilbert space of 
$$\pi_s:=\Ind_{H \cap s^{-1}H s}^H \rho_s,$$
with the following properties:
\begin{itemize}
\item $\sum_{s\in G \setminus H} \Vert F_n^s\Vert^2=1$ for all $n\geq 1;$
\item $\lim_{n\to +\infty} \sum_{s\in G \setminus H} \Vert  \pi_s(h) F_n^s- F_n^s\Vert^2=0$
for every $h\in H.$ 
\end{itemize}
Define, for every $s\in   G \setminus H,$ a sequence $(f_n^s)_{n\geq 1}$
of functions $f_n^s : H\to \mathbf{R}$ by 
$$f_n^s(x)= \Vert F_n^s(x)\Vert \qquad\text{for} \quad x\in H.$$
Then $f_n^s$ is constant on the $H \cap s^{-1}H s$-cosets and, for the norm of $f_n^s$ in  $\ell^2(H/(H \cap s^{-1}H s)),$ we have
$$\Vert f_n^s\Vert^2= \sum_{x\in H/ (H \cap s^{-1}H s)}  \Vert F_n^s(x)\Vert^2= \Vert F_n^s\Vert^2;$$ 
hence, 
$$
\sum_{s\in G \setminus H}  \Vert f_n^s\Vert^2= \sum_{s\in G \setminus H} \Vert F_n^s\Vert^2=1
$$ 
for all $n\geq 1.$ Moreover, for $s\in G \setminus H$ and $h\in H$, we have 
$$
\begin{aligned}
\Vert (\lambda_{H/ (H \cap s^{-1}H s)})(h)f_n^s- f_n^s\Vert^2&= 
\sum_{x\in H/ (H \cap s^{-1}H s)} \left\vert\Vert F_n^s(h^{-1} x)\Vert  - \Vert F_n^s( x)\Vert \right\vert^2\\
&\leq \sum_{x\in H/ (H \cap s^{-1}H s)} \Vert F_n^s(h^{-1} x)- F_n^s(x)\Vert^2\\
&= \Vert  \pi_s(h) F_n^s- F_n^s\Vert^2.\\
\end{aligned}
$$
Therefore,
$$\lim_{n\to +\infty} \sum_{s\in G \setminus H} \Vert (\lambda_{H/ (H \cap s^{-1}H s)})(h)f_n^s- f_n^s\Vert^2=0.$$
and so 
$$
f_n:=\bigoplus_{s\in G \setminus H}f_n^s 
$$
is a sequence of almost invariant unit vectors for the representation 
$$\bigoplus_{s\in G \setminus H} \lambda_{H/ (H \cap s^{-1}H s)};$$
this contradicts the fact that $H$ has the spectral gap property (see 
Remark~\ref{Rem-Def-SpectralSubgroups}.ii) and proves the claim.

 \vskip.5cm
\noindent
$\bullet$ \textbf{Step 2.} 
Set  $\pi:=\Ind_H^G \sigma$ for $\sigma\in \widehat{H}_{fd}$. We claim that $\pi$ is an irreducible and traceable representation of $G.$
 
 Indeed, since $H$ is self-commensurating, the irreducibility of $\pi$ follows from Mackey's theorem.
 To prove that $\pi$ is traceable, it suffices to show that $\pi(C^*(G))$ contains a non-zero compact operator.
 
 Let $S\subset G$ be  a system of representatives for  the double coset space $H\backslash G/H$
with $e\in S$.
By Lemma~\ref{Lemma-Induced}, $\pi|_H$
is equivalent to 
$$
 \bigoplus_{s\in S} \Ind_{H \cap s^{-1}H s}^H (\sigma^s|_{H \cap s^{-1}H s}) = \sigma \oplus \bigoplus_{s\in S\setminus \{e\}} 
  \Ind_{H \cap s^{-1}H s}^H (\sigma^s|_{H \cap s^{-1}H s}).
  $$
  By Step 1 above, the representation
  $$\bigoplus_{s\in S\setminus \{e\}}  \Ind_{H \cap s^{-1}H s}^H (\sigma^s|_{H \cap s^{-1}H s}).$$
  does not weakly contain the finite dimensional representation $\sigma$.
  It follows from \cite[Lemma 6]{Bekka-GLnQ} that $\pi(C^*(H))$ contains a non-zero compact operator.
  Since $C^*(H)$ can be viewed as subalgebra of $C^*(G),$ this proves  the claim and hence Item (i).
 \end{proof}

\subsection{Kazhdan subgroups}
\label{SS-Kazhdan}
 We now turn to examples of  subgroups with the spectral gap property.
We need the following  strengthening of the notion of a self-commen\-surating subgroup.
\begin{definition}
\label{Def-StronglySelfComm}
A subgroup $H\in \Sub(G)$  is \textbf{strongly self-commen\-surating} 
 in $G$ if $gH g^{-1}\cap H$ has infinite index in $H$ for every $g\in G\setminus H.$ 
\end{definition}

\begin{remark}
\label{Rem-SelfNorm}
\begin{itemize}
\item[(i)] It follows from Remark~\ref{Rem-Def-SpectralSubgroups}.i that 
every subgroup $H$ in $\Sub_{\rm sg}(G)$  is  strongly self-commen\-surating in $G.$
\item[(ii)]  A self-commensurating subgroup $H\in \Sub(G)$ is not necessarily strongly self-commen\-surating.
Indeed, let
 $G=F_2$ be the free group on two generators $a$ and $b.$
 The subgroup $H$ generated by $\{b^{n} a b^{-n}\mid n \geq 1\}$
  is easily seen to be self-commensurating. However, $b^{-1}Hb \cap H=H.$ 
  \end{itemize}
     \end{remark}

Our first  main class of examples of  subgroups with the spectral gap property is given by  subgroups with Kazhdan's property (T);
for an account on Kazhdan's property (T), see \cite{BHV}.

 \begin{proposition}
\label{Prop-StronglySelfCommKazhdan}
Let $H\in \Sub(G)$ be a strongly self-commen\-surating subgroup with property (T).
Then $H\in \Sub_{\rm sg}(G)$.
\end{proposition}

\begin{proof}
Assume, by contradiction, that  
the action $H\curvearrowright X$ does not have a spectral gap, where $X= G/H \setminus \{H\}$.
Then (see Remark~\ref{Rem-Def-SpectralSubgroups}.ii) $1_H$ is weakly contained  in  the direct sum
$$
\bigoplus_{s\in G \setminus H} \lambda_{H/ H \cap s^{-1}H s}.
$$
Since $H$ has property (T), this  implies $1_H$ is contained in $\lambda_{H/ H \cap s^{-1}H s},$
for some $s\in G \setminus H.$ Then  $H \cap s^{-1}H s$ has finite index in $H$ and this 
contradicts the fact that $H$ is strongly self-commensurating.
\end{proof}

It follows from Proposition~\ref{Prop-StronglySelfCommKazhdan} and Theorem~\ref{Theo-RepRigidSubgroups}
that  strongly self-commen\-surating  Kazhdan subgroups  are strongly representation rigid inside the class 
 $\Sub_{sc}(G)$ of self-commensurating subgroups.
We do not know whether such  subgroups are representation rigid 
(or strongly representation rigid) inside the whole of $\Sub(G).$ 
However, as we now show, they are conjugation rigid.
Indeed, as Kazhdan groups are finitely generated, this follows from the following 
proposition.
\begin{proposition}
\label{Prop-RigidFinGenSubgroups}
Let  $H$ be a strongly self-commensurating  finitely generated subgroup of $G.$
Then $H$ is conjugation rigid: if $L\in \Sub(G)$  is such that 
 $\overline{\C(L)}= \overline{\C(H)},$ then $L$ is conjugate to $H.$
\end{proposition} 
\begin{proof}
As $\overline{\C(L)}= \overline{\C(H)},$ we have 
 $\lim_n x_n L x_n^{-1} =H$ and $\lim_n y_n H y_n^{-1} =L$ for sequences 
 $(x_n)_n, (y_n)_n$ in $G.$
 Since $H$ is finitely generated, it follows from the definition of the Chabauty topology
that there exists $n_0\geq 1$ such that   $H\subset x_nLx_n^{-1}$ for  every $n\geq n_0.$
As 
$$\lim_n x_{n_0} y_n H y_n^{-1} x_{n_0}^{-1}=x_{n_0}Lx_{n_0}^{-1}$$  and, again, as
$H$ is finitely generated, there exists $n_1\geq 1$ such that 
$$H\subset x_{n_0} y_n H y_n^{-1} x_{n_0}^{-1} \qquad\text{for all} \qquad n\geq n_1.$$
Since $H$ is strongly self-commensurating, it follows that $x_{n_0} y_n \in H$, that is, 
$x_{n_0} y_n H y_n^{-1} x_{n_0}^{-1}= H,$ for all $n\geq n_1.$
 Hence, $x_{n_0}Lx_{n_0}^{-1}=H.$
  \end{proof}

\begin{remark}
\label{rem-KazhdanSubNonClosed}
The set of Kazhdan subgroups  is in general neither closed nor open 
in $\Sub(G),$ as the following examples show.
\begin{itemize}
\item Let $G=\bigcup_{n\in \NN} H_n$ be an inductive limit of 
a strictly increasing family of finite subgroups $H_n$.
Then every $H_n$ has property (T) but $\lim_nH_n=G$ does not have  property (T),
as it is not finitely generated. 
\item The sequence of the subgroups $H_n=2^n \ZZ$ of $G=\ZZ$,
which  of course do not have property (T), converges to the Kazhdan subgroup $\{e\}$.
\end{itemize}
\end{remark}

\subsection{A-normal subgroups}
\label{SS:A-Normal}
We are going to introduce our second class of examples of subgroups with the spectral gap property.

Recall that a subgroup $H$ of $G$ is said to be \emph{malnormal} 
(respectively \emph{weakly malnormal})  if $H\cap gH g^{-1} =\{e\}$ 
(respectively $H\cap gH g^{-1}$ is finite) for every $g\in G\setminus H.$
For the relevance of this notion in group theory and operator 
algebras, see \cite{Lyndon-Schupp} and \cite{Popa-Vaes}).
We introduce a class of subgroups which contains  all weakly malnormal subgroups.

\begin{definition}
\label{Def-A-normalSubgroups}
A subgroup $H\in\Sub(G)$ is said to be \textbf{a-normal}  if 
$H\cap gH g^{-1}$ is amenable for every $g\in G\setminus H.$
We denote by $\Sub_{\rm a-norm}(G)$ the set of all a-normal subgroups of  $G.$
\end{definition}

\begin{remark}
\label{Rem-Def-A-normalSubgroups}
\begin{itemize}
\item[(i)] Of course, every malnormal subgroup and every amenable subgroup is a-normal; however,
we will be mainly interested in a-normal subgroups which are not amenable.
\item[(ii)] Every {\it non-amenable} a-normal subgroup $H$ of $G$ is  strongly self-commen\-surating in $G$ in the sense of Definition~\ref{Def-StronglySelfComm}:  
for  $g\in G\setminus H$, the group  $H\cap gH g^{-1}$ is   amenable  and so cannot have finite index in the non-amenable group $H.$

\end{itemize}
\end{remark}

 \begin{proposition}
\label{Prop-ANormalSpectralGap}
Let $H$ be a non-amenable a-normal subgroup of $G$.
Then $H\in \Sub_{\rm sg}(G)$.
\end{proposition}

\begin{proof}
Assume by contradiction that $H$ is not in $\Sub_{\rm sg}(G)$, that is (see Remark~\ref{Rem-Def-SpectralSubgroups}), 
$1_H$ is weakly contained  in  the direct sum
$$
\bigoplus_{s\in G \setminus H} \lambda_{H/ H \cap s^{-1}H s}.
$$
Since $H$ is a-normal, $H \cap s^{-1}H s$ is amenable and so $\lambda_{H/ H \cap s^{-1}H s}$ is weakly contained in $\la_{H}$ 
for every $s\in G \setminus H$ (see Corollary~\ref{Cor-AmenableInduced}).
This implies that $1_H$ is weakly contained in $\la_{H}$, hence that $H$ is amenable, and this is a contradiction.
\end{proof}

It follows from Proposition~\ref{Prop-ANormalSpectralGap} and Theorem~\ref{Theo-RepRigidSubgroups}
that non-amenable subgroups in $\Sub_{\rm a-norm}(G)$ are strongly representation rigid inside the class 
 $\Sub_{sc}(G)$ of self-commensurating subgroups.
 
We do not know whether non-amenable a-normal subgroup are representation rigid 
(or  strongly representation rigid) inside the whole of $\Sub(G).$ 
However, as we now show, they are conjugation rigid inside  $\Sub(G).$ 
For this, we will need the following observation.

\begin{lemma}
\label{Lem-A-NormalSubg}
The set $\Sub_{\rm a-norm}(G)$ is a closed and $G$-invariant subset of $\Sub(G).$
\end{lemma}

\begin{proof}
Since   the conjugate of an a-normal subgroup is obviously again a-normal,
$\Sub_{\rm a-norm}(G)$ is $G$-invariant.

Let $(H_n)_n$ be  a sequence in  $\Sub_{\rm a-norm}(G)$ converging to $H\in \Sub(G).$ 
Let $x\in G\setminus H.$  Since $$\lim_{n} H_n \cap xH_n x^{-1} =H \cap xH x^{-1}$$
and since,  by Corollary~\ref{Cor-AmenableSubgroups},  $\Sub_{a}(G)$ is closed in $\Sub(G)$,
it suffices to show that $x\notin H_n$ for $n$ large enough.

Assume, by contradiction, that $x\in H_{n_k}$ for a subsequence $(H_{n_k})_k$ of $(H_n)_n.$
Then $x\in H$, by the definition of the Chabauty topology, and this is a contradiction.
\end{proof}

\begin{corollary}
\label{Cor1-Theo-RepRigidSubgroups}
Let $H$ be a non-amenable a-normal subgroup of $G.$ Then $H$ is conjugation rigid: 
if $L\in \Sub(G)$ is such that $\overline{\C(L)}= \overline{\C(H)},$ then $L$ is conjugate to $H.$
\end{corollary} 

\begin{proof}
Let  $L\in \Sub(G)$ be such that $\overline{\C(L)}= \overline{\C(H)}.$
Then $L$ is a-normal, by Proposition~\ref{Lem-A-NormalSubg}.
We claim that $L$ is non-amenable. Indeed, since $H\in \overline{\C(L)}$, it would  follow
otherwise that $H$ is amenable (see Corollary~\ref{Cor-AmenableSubgroups}) and this contradicts our assumption.
The claim now follows from Corollary~\ref{Cor1} and Theorem~\ref{Theo-RepRigidSubgroups}.
\end{proof}

\section{Point-stabilizers of topological actions}
\label{S:Actions}
In this section, we study classes of subgroups of a group $G$ which arise as  stabilizers of points for 
certain actions $G\curvearrowright X$ on a topological space $X.$
\subsection{Continuity points of the stabilizer map}
\label{SS:Actions}
Given an action $G\curvearrowright X$  of $G$ on a Hausdorff topological space $X$
by homeomorphisms, let $\Stab: X\to \Sub(G)$ be the $G$-equivariant map  defined by 
$$
\Stab(x)= G_x \qquad \text{for}\qquad x\in X,
$$
where $G_x = \{g\in G : gx=x\}$ is the stabilizer of $x$ in $G.$ 

For $x\in X,$ denote by $G_x^0$ the (normal) subgroup of $G_x$ consisting of all $g\in G$
for which there exists a neighbourhood $U_g$ of $x$ such that $gy=y$ for all $y\in U_g.$ 
It is easy to show that $\Stab: X\to \Sub(G)$  is continuous at $x\in X$ if and only if 
$G_x=G_x^0$; moreover, the set  of continuity points of this map
is a dense $G_\delta$-subset of $X$ if $X$ is a Baire space (for the elementary proof of these facts,  see Lemmas 2.2 and Proposition 2.4 in \cite{LeM}).

The following result, announced in Proposition~\ref{Pro-Def-RepEquivalentSub}, 
shows that, in particular,   the equivalence relations 
$\sim_{\rm w-conj}$ and  $\sim_{\rm rep}$ on  $\Sub(G)$  are not comparable in general.

Recall (see \cite{CFP}) that Thompson's group $T$ is the group of orientation preserving homeomorphisms
of  the circle $\mathbf{S}^1=\mathbf{R}/\mathbf{Z},$ which are piecewise linear, with only
finitely many breakpoints, all at dyadic rationals, and with  slopes all of the form  $2^k$ for $k\in \mathbf{Z}.$
Thompson's group $F$ is the stabilizer in $T$ of the point $1\in \mathbf{S}^1.$

As is well-known and easy to check, the action $T\act  \mathbf{S}^1$ is minimal.
Recall that an action $G\act X$ of a group $G$  on a  topological space $X$
is said to be \textbf{minimal} if the $G$-orbit $Gx$ is dense in $X$ for every $x\in X$. 

\begin{theorem}
\label{thm:Furst-bnd-non-C*-simple}
Let $G$ be Thompson's group $T$ and consider its action on $\mathbf{S}^1.$  Let $x, y \in \mathbf{S}^1$
and let $\sigma$ and  $\rho$ be finite dimensional irreducible unitary representations of $G_x$ and $G_y$, respectively.

\begin{itemize}
\item[(i)]  The representation $\Ind_{G_x}^G\sigma$ is irreducible.
\item[(ii)]  The representations $\Ind_{G_x}^G \sigma$ and $\Ind_{G_y}^G \rho$  are equivalent if and only if  there exists $g\in G$ such that $gx=y$ and $\rho$ is equivalent to $\sigma^{g^{-1}}$.
 \item[(iii)] Assume that $y\notin Gx;$ then  $G_x \not\sim_{\rm rep} G_y.$
 \item[(iv)]  Assume that $x,y\in  \mathbf{S}^1\setminus \{e^{2\pi i\theta} \mid \theta\in \QQ\}$; then  $G_x \sim_{\rm w-conj} G_y$.
\end{itemize}
\end{theorem}
\begin{proof}
Assume that $x\neq y.$  We claim that  $G_x \cap G_y$ has infinite index in $G_x$ or, equivalently, 
that  the $G_x$-orbit of $y$ is infinite. Indeed, it follows from Lemma 4.2 in \cite{CFP}  that every  orbit  
of $G_1\cap G_x=F\cap G_x$  in $\mathbf{S}^1\setminus \{x, 1\}$ is infinite and it is clear that the $G_x$-orbit of $1$
is infinite if $x\neq 1.$

Let $g\in G\setminus G_x.$ Then, by what we have just shown,
$G_x\cap gG_x g^{-1}= G_x\cap G_{gx}$ has infinite index  in $G_x.$ Hence, $\Omm_G(G_x)=G_x$.

Moreover, we have $G_y= gG_x g^{-1}$ for some $g\in G$ if and only if $G_y=G_{gx}$, that is, if and only if
$y=gx$. So, Items (i) and (ii) follow from Mackey's Theorem \ref{Theo-Mackey}.

Item (iii) is a direct consequence of Item (ii).

To prove Item (iv), note that $\mathbf{S}^1\setminus \{e^{2\pi i\theta} \mid \theta\in \QQ\}$ is contained in the set $C$ of continuity points of the map $\Stab: \mathbf{S}^1\to \Sub(G)$.
 
 Let  $\mathcal{S}_G( \mathbf{S}^1)$ be the closure of  $\{G_z\mid z\in C\}$  in $\Sub(G)$.
Since $T\act  \mathbf{S}^1$ is minimal,  the action  $G\act\mathcal{S}_G( \mathbf{S}^1)$ is minimal,  by \cite[Proposition 1.2]{GW}.
 It follows that $G_x \sim_{\rm w-conj} G_y$ for all $x, y\in C$.
\end{proof}

\begin{remark}
\label{Rem-URS}
Consider the action of Thompson group $T$ on $\mathbf{S}^1$;
the closure in $\Sub(T)$ of the set $\{T_z\mid z\in C\}$, 
which appeared in the proof of Theorem~\ref{thm:Furst-bnd-non-C*-simple},
is the stabilizer URS (short for ``uniformly recurrent subgroup") 
for  $T\act \mathbf{S}^1$ in the sense of \cite{GW};
 this URS has been completely described in \cite[Proposition 4.10]{LeM}.
 \end{remark}

We apply now  the result (Proposition~\ref{Pro:FellChabautyTop}) on the continuity of the
map 
$$
\Lambda: \Sub(G) \to \Rep (G), \qquad H\mapsto \la_{G/H}
$$ 
 to point-stabilizers of a topological  action of a group $G.$

\begin{proposition}
\label{Cor1-Action}
Let $G\curvearrowright X$ be an action  of $G$ on a Hausdorff topological space $X.$
\begin{itemize}
\item[(i)] Let $x\in X$ and let $z\in \overline{G x}$ be such that $G_z=G_z^0$.
Then $G_z$ belongs to the closure of the conjugacy class $\C(G_x)$
of $G_x$; in particular, $\la_{G/G_z} \prec \la_{G/G_x}$.
\item[(ii)] Let  $x\in X$ be such that $ \overline{G x}$
contains a point with trivial stabilizer and $H$ be a subgroup of $G_x.$
Then $\{e\}$ belongs to the closure 
of $\C(H)$ and $\la_{G} \prec \la_{G/H}$. 
\end{itemize}
\end{proposition}
\begin{proof}
For Item (i), note that by assumption, $\Stab: X\to \Sub(G)$  is continuous at $z$ and $z\in \overline{G x}$. Thus, we have $G_z\in \overline{\C(G_x)}$, and Proposition~\ref{Pro:FellChabautyTop} implies that $\la_{G/G_z} \prec \la_{G/G_x}.$  

To show Item (ii), let $z\in \overline{G x}$ be  such that $G_z=\{e\}.$ We have trivially $G_z=G_z^0$; hence, $\{e\}\in \overline{\C(G_x)},$ by Item (i).
It follows from the definition of the Chabauty topology that  $\{e\}\in \overline{\C(H)}$
 for every subgroup $H$ of $G_x$. This, together with   Proposition~\ref{Pro:FellChabautyTop},
 implies that $\la_{G} \prec \la_{G/H}$.
\end{proof} 

Recall that an action $G\act X$ of $G$ on a Hausdorff topological space $X$  is called \textbf{topologically free} if, for every  $g \in G\setminus \{e\},$ the set  ${\rm Fix}(g)=\{x\in X \mid gx=x\}$ of fixed points of  $g$ has empty interior in $X$.
Observe that,  for such an action, the set of points $x\in X$ with $G_x=\{e\}$ is a dense (and in particular, non empty) 
subset of $X,$ provided $X$ is  a Baire space; indeed,  the set of continuity points of the stabilizer map $X\to \Sub(G)$ is 
dense, as mentioned at the beginning of this section.

A Hausdorff topological space $X$ is called \emph{extremally disconnected} (or \emph{Stonean}) if the closure of every open set of $X$ is open.  A relevant fact about these spaces is the following result from  \cite{F1971}.

\begin{lemma}
\label{Lem-ExtremallyDisconnected}
Let $X$ be an  extremally disconnected topological space. The fixed point set of every  homeomorphism of 
$X$ is open.  In particular, if $G\act X$ is an action of a group $G$ on $X$, then $G_x = G_x^0$ for every  $x\in X$
and  so the map $\Stab: X\to \Sub(G)$ is continuous. 
\end{lemma}

The following corollary is an immediate consequence of Lemma~\ref{Lem-ExtremallyDisconnected} and  Proposition~\ref{Cor1-Action}.

\begin{corollary}\label{prop:ext-discon}
Let $G\act X$ be a minimal  action of $G$ on an extremally disconnected topological space $X$. Then 
$G_x \sim_{\rm w-rep} G_y$ for all $x, y \in X$.
\end{corollary}

\subsection{Boundary actions}
\label{SS-BoudaryActions}
We turn our attention to subgroups which are point-stabilizers of boundary actions of a countable group $G.$

Let $X$ be a compact topological space.
Recall that an action  $G\act X$   of $G$ on  $X$ is said to be \textbf{strongly proximal} if, for every probability $\nu\in\pr(X)$ on $X$,  the weak* closure of the orbit $G\nu$ contains some point measure $\delta_x$, $x\in X$.
The action $G\act X$ is called a \textbf{boundary action} or a \textbf{$G$-boundary} if it is both minimal and strongly proximal.

By \cite{Furstenberg}, every group $G$ admits a universal boundary $\partial_F G$, called the \textbf{Furstenberg boundary}:
 $\partial_F G$ is a  $G$-boundary and every $G$-boundary is a continuous $G$-equivariant image of $\partial_F G$.
 Moreover, it is known that  $\partial_F G$ is extremally disconnected (see \cite[Remark 3.16]{KalKen} or \cite[Proposition 2.4]{BKKO}).

 It follows from Lemma~\ref{Lem-ExtremallyDisconnected}
that, if $X$ is extremally disconnected, then $G\act X$ is topologically free
if and only if $G\act X$ is \textbf{free}, that is, ${\rm Fix}(g)=\emptyset$ for every $g \in G\setminus \{e\}$.

$C^*$-simple groups have been characterized in \cite{KalKen} in terms  their actions on boundaries 
as follows (see also \cite{BKKO}).
\begin{theorem} \label{Theo-KK} 
\cite{KalKen}
Let $G$ be a discrete group. The following properties are equivalent:
\begin{itemize}
\item[(i)] $G$ is $C^*$-simple;
\item [(ii)] the action $G\act \partial_F G$  is  free;
\item[(iii)] there exists a topologically free boundary  action $G\act X$.
\end{itemize}
\end{theorem} 

 Recall that the \textbf{amenable radical}  of a group $G$ is the unique largest amenable normal subgroup of $G.$
 A  long   well-known fact  (see \cite[Lemma p.289]{BCH}) is that a group with a non-trivial amenable radical is not $C^*$-simple;
 examples of  non $C^*$-simple groups with a trivial amenable radical have been given in \cite{LeBoudec}.
\begin{proposition}
\label{prop:Furst-bnd-non-C*-simple}
Let $G$ be a  countable group which is not $C^*$-simple  and let $G\act \partial_F G$ be the Furstenberg boundary action. 
\begin{itemize}
\item[{\normalfont (i)}] For  every $x \in \partial_F G$, the stabilizer $G_x$ is non-trivial and amenable.
\item[{\normalfont (ii)}] For any $x, y \in \partial_F G$, we have $G_x \sim_{\rm w-conj} G_y .$
\item[{\normalfont (iii)}] For every $x \in \partial_F G$, we have $
G_x \not\sim_{\rm w-rep} \{e\}.$
\item[{\normalfont (iv)}]  Assume that the amenable radical  of $G$ is trivial. Then there exists
an uncountable subset $A$ of $ \partial_F G$ such that, for every $x,y\in A$ with $x\neq y,$ we have $G_x \not\sim_{\rm conj} G_y.$
\end{itemize}
\end{proposition}

\begin{proof}
Since $G$ is not $C^*$-simple, Item (i)  follows from Theorem 6.2 in \cite{KalKen}  and  from its proof
(see also Theorem 3.1 and Proposition 2.7 in \cite{BKKO}). 

As $\partial_F G$ is extremally disconnected,  the map $\Stab: \partial_F G\to \Sub(G)$ is continuous
(see Lemma~\ref{Lem-ExtremallyDisconnected}).
Since $G\act \partial_F G$ is minimal, the action of $G$ on $\{G_x \mid x\in  \partial_F G\}$
is also minimal.  This implies that $G_x \sim_{\rm w-conj} G_y$, for any pair of points $x, y \in \partial_F G$, and therefore proves Item (ii). 

Let $x\in \partial_F G$. Since $G_x$ is non-trivial, it follows from \cite[Proposition 3.5]{BKKO} that  $\la_G\not\prec \la_{G/G_x}$ and hence $G_x \not\sim_{\rm w-rep} \{e\}$. This proves Item (iii).

By Item (i) and \cite[Proposition 2.18(iii)]{LeM}, the set $\{G_x \mid x\in  \partial_F G\}$ is uncountable,
provided $G$ has a trivial amenable radical.  Since $G$ is countable, this implies Item (iv). 
\end{proof}

\subsection{A criterion for $C^*$-simplicity}
We give a sufficient condition for the  $C^*$-simplicity of a countable group $G$ in terms of a dynamical property of \textbf{non necessarily compact} $G$-spaces.

\begin{proposition}\label{thm:non-top-free-stab}
Let $G\act Y$ be a non-topologically free boundary action of $G$, and let $G\act X$ be a topologically free minimal action of $G$ on a Baire  space $X$. For every $y\in Y$, the subgroup $G_y$ has no global fixed point in $X$.
\end{proposition}

\begin{proof}
Since $G\act Y$ is a non-topologically free boundary action, the quasi-regular representation $\lambda_{G/G_y}$ does not weakly contain the regular representation $\lambda_G$ for any $y\in Y$,  by \cite[Proposition 3.5]{BKKO}.

Assume, by contradiction, that $G_y$ fixes a point $x\in X.$
Since  $G\act X$ is  minimal and topologically free, $\overline{Gx}=X$ contains
a point with trivial stabilizer.
Proposition \ref{Cor1-Action} implies that 
 $\la_{G} \prec \la_{G/G_y}$ and this is a contradiction.
\end{proof}

\begin{example}
\label{Exa-Thompson}
Let $T$ and $F$ be the Thompson groups.
\begin{itemize}
\item[(i)] The standard action $T\act \mathbf{S}^1$
is a non-topologically free boundary action and  $F$ is a point-stabilizer for this action. 
Let $T\act X$ be  a topologically free minimal action of $T$ on a Baire space $X$.
It follows from  Proposition \ref{thm:non-top-free-stab} that   $F$ has no global fixed point in  $X$.
\item[(ii)]  Let $T\act K$  be a minimal action of $T$ on a compact space $K$ such that   $F$ stabilizes a point in $K$.
The  rigidity result on minimal actions of $T$ on compact spaces from  \cite[Theorem 1.8]{LeM}, combined with (i),
shows that  $T\act K$  factors onto the standard action $T\act \mathbf{S}^1.$
\end{itemize}
\end{example}

As a consequence of Proposition \ref{thm:non-top-free-stab}, we obtain  a sufficient  condition for $C^*$-simplicity of a group.

\begin{corollary}\label{cor:amen-fix->C*-simple}
Let  $G$ be  a group with the following property: for every amenable subgroup $H\in \Sub(G)$, there exists  a topologically free minimal action $G\act X$ on a Baire  space $X$ such that $H$ fixes a point in $X$. Then $G$ is $C^*$-simple.
\end{corollary}

\begin{proof}
For each $y\in \partial_FG$, the stabilizer $G_y$ is amenable and hence fixes a point $x\in X,$ by  assumption. 
Proposition~\ref{thm:non-top-free-stab} implies that $G\act\partial_FG$ is topologically free;  therefore, $G$ is $C^*$-simple
(see Theorem~\ref{Theo-KK}).
\end{proof}

\begin{example}
Let $G$ be a torsion-free word hyperbolic group and let $G\act\partial G$ be its the action on its Gromov boundary. Then every amenable subgroup of $G$ is infinite cyclic (see e.g. \cite[Theorem 12.2]{KapBen}) generated by a loxodromic element, hence fixes two points in $\partial G.$
\end{example}


\subsection{Examples of a-normal subgroups}
\label{SS:ExaANormal}
In this subsection, we give  some  examples of classes of a-normal subgroups as defined in Subsection~\ref{SS:A-Normal}.

\begin{example}
\label{Ex:free-pr-a-nor}
Let $G = *_i H_i$ be a free product of countable discrete groups $H_i$. It follows from the definition of the free product that, for 
every  factor group $H_i$ and every $s\notin H_i$, we have $sH_is^{-1}\cap H_i = \{e\}$. Hence, every  factor $H_i$ is an a-normal subgroup of $G$. If, in addition, a factor group $H_i$ is non-amenable, then $H_i\in \Sub_{\rm sg}(G)$.
\end{example}

We can determine when  two factor groups in a  free product are weakly representation equivalent.

\begin{proposition}
\label{Pro-FreeProduct}
Let $G= H_1 * H_2$ be the free product of non-trivial countable groups $H_1$ and $H_2$, at least one of which is of order bigger than two. Then $H_1\sim_{\rm w-rep} H_2$ if and only if both $H_1$ and $H_2$ are amenable.
\end{proposition}

\begin{proof}
The group $G$ is $C^*$-simple (see  for instance \cite{PaSa--79}). 
Assume that $H_1$ and $H_2$ are amenable.  Then, by \cite[Theorem 1.2]{BKKO}, we have
$$\lambda_{G/H_1}\sim \lambda_{G} \sim\lambda_{G/H_2}.$$

Conversely, assume that $H_1$ is non-amenable. Since $H_1$ is a-normal (see Example~\ref{Ex:free-pr-a-nor}), it follows $H_1$ is conjugation rigid, by Corollary \ref{Cor1-Theo-RepRigidSubgroups}. On the other hand, $H_1\not\sim_{\rm conj}H_2$, by definition of the free product. Hence,  $\lambda_{G/H_1}\not\sim_{\rm w-rep} \lambda_{G/H_2}$.
\end{proof}

The subgroups $H_i$ from Proposition~\ref{Pro-FreeProduct}
are  examples   of maximal parabolic subgroups of relatively hyperbolic groups, in the sense of Bowditch \cite{Bow12}, which in turn are special cases 
of convergence groups (see \cite{Yaman} and \cite{Bow98} for more details).

\begin{definition}
\label{parabolic-subgroup}
\begin{itemize}
\item[(i)] A discrete group $G$ is a \emph{convergence group} if it admits an action $G\act X$ by homeomorphisms on a perfect compact Hausdorff space such that the induced action on the set of distinct triples $\{(x, y, z)\in X\times X\times X \,\mid\, \text{Card}\{x, y, z\} =3 \}$ is properly discontinuous. In this case, we say that $G\act X$ is a \emph{convergence action}.
\item[(ii)] Let $G$ be a convergence group. We say that a subgroup $H\in \Sub(G)$ is \emph{parabolic} if it is infinite and if there is a convergence action $G\act X$ such that $H$ fixes some point of $X$ and such $H$ contains no \emph{loxodromic} element,
that is, no element $g\in G$ of infinite order with $\text{Card}({\rm Fix}(g)) = 2$.
 \end{itemize}
\end{definition} 

If $G$, $H$, and $G\act X$ are as in the above Definition \ref{parabolic-subgroup}.ii, then $H$ has a unique fixed point
in $X$  which is called a \emph{parabolic point}. The stabilizer of a parabolic point is a parabolic subgroup and  hence a maximal parabolic subgroup.

\begin{proposition}
\label{Pro-ANormalParabolic}
Let $G$ be a convergence group. 
Let $H\in\Sub(G)$ be a torsion-free maximal parabolic subgroup. Then  $H$ is malnormal and hence a-normal. If, moreover, 
$H$ is non-amenable, then  $H\in \Sub_{\rm sg}(G).$
\end{proposition}

\begin{proof}
Since $H\in\Sub(G)$ is a maximal parabolic subgroup, there is a convergence action $G\act X$ as in the Definition \ref{parabolic-subgroup}.ii, and a parabolic point $x\in X$ such that $H=G_x$. Let $g\in G\setminus H$, and let $s\in H\cap gHg^{-1}$. Then $s$ fixes the pair of distinct points $x$ and $gx$. Since $H$ does not contain any loxodromic elements, it follows $s$ has finite order and so $s=e,$ as $H$ is torsion free. Thus, $H\cap gHg^{-1}=\{e\}.$
\end{proof}

Here is another class  of actions for which point-stabilizers  are  a-normal (and not necessarily malnormal).

\begin{proposition}\label{prop:anormal-stab}
Let $G\act X$ be an action of the  group $G$ on a set $X.$
Assume that no non-amenable subgroup of $G$ fixes more than one point in $X$. Then $G_x$ is a-normal for every $x\in X$.
\end{proposition}

\begin{proof}
Let $x\in X$ and $g\in G\setminus G_x$. Then $gx\neq x$; the subgroup $G_x\cap gG_x g^{-1}$, which fixes both $x$ and $gx$, is therefore amenable.
\end{proof}

\begin{example}
\label{Exa-A-NormalSubg} 
Let $\KK$ be a field and $G$ a subgroup of $GL_3(\KK).$
The standard action $G\act X$ on the  projective plane $X=\PP^2(\KK)$ satisfies the condition of Proposition \ref{prop:anormal-stab} above; indeed, for distinct points $x$ and $y$ in $X$, the group $G_x\cap G_y$ is conjugate inside $GL_3(\KK)$ to a subgroup of the group
\[ 
L:= \left[\begin{array}{ccc}
* & 0 & *\\
0 & * & *\\
0 & 0 &*
\end{array}\right];\]
since  $L$  is amenable, it follows that $G_x\cap G_y$ is amenable. Hence $G_x$ is a-normal for every point $x\in X$.  
\end{example}

\section{Weakly parabolic subgroups}
\label{S:StabilizerProba}
In this section, we study subgroups of a group $G$  which stabilize a probability measure on a $G$-boundary.

\subsection{Weakly parabolic subgroups: definition and examples}
\label{SS:WPS-DefExa}
Let $X$ be a compact space. Recall  that the space $\pr(X)$ of probability measures on $X,$
endowed with the weak* topology, is a compact space.
Assume that we have an action   $G\act X$ of a group $G$   on $X$. Then 
$G$ acts on $\pr(X)$ through $\nu\mapsto g_*\nu,$
where $g_*\nu$ is the image of $\nu\in \pr(X)$ under the map $x\mapsto gx$ for $g\in G.$

We introduce our second main  class of subgroups of $G.$
\begin{definition}
\label{Def-Sub_g}
Let $G$ be a countable group.  We say that a subgroup $H$ of $G$ is \emph{weakly parabolic}, if
there exists a topologically free boundary action $G\act X$ such that $H$ has a fixed point in $\pr(X).$
We denote by $\Sub_{\rm w-par}(G)$ the set of weakly parabolic $H\in \Sub(G)$.
\end{definition}

\begin{remark}
\label{Rem-Sub_g} 
Let $G$ be a group which admits  a topologically free boundary action $G\act X$.
\begin{itemize}
\item[(i)] The group $G$  is $C^*$-simple, by Theorem~\ref{Theo-KK}.
\item[(ii)] The set 
$$ \left\{H\in \Sub(G)\,\mid\, \text{$H$ has a fixed point in $\pr(X)$}\right\}$$ 
 is a closed and $G$-invariant subset of $\Sub(G)$.
 It follows that $\Sub_{\rm w-par}(G)$ is a $G$-invariant subset of $\Sub(G)$.
\end{itemize}
\end{remark}

\begin{example}
\label{Exa-Sub_g} 
(i) Let $G$ be a group which admits  a topologically free boundary action.
 Then every amenable subgroup of $G$ belongs to $\Sub_{\rm w-par}(G)$.
Indeed, if $H\act X$ is an action of an amenable group
on a compact space $X,$ then $H$ has a fixed point in $\pr(X)$.
 
 \noindent 
 (ii)  Let $G\act X$ be a topologically free boundary action of a group $G.$
  Let $x\in X$. Every subgroup of the point-stabilizer $G_x$ of $x$ is weakly parabolic.
Here is a specific example.

\begin{itemize}
\item Let $G=PGL_n(\QQ)$ for $n\geq 2$ and let $G\act X$ be the standard action
 on the projective space $X=\PP(\kk^n)$, where $\kk$ is the field $\RRR$ of real numbers
or the field $\QQ_p$ of $p$-adic numbers.  
Let $x$ be the image in $\PP(\kk^n)$ of some vector from $\QQ^n\setminus\{0\}.$
Then the point stabilizer $G_x$ is isomorphic to 
the semi-direct product $PGL_{n-1}(\QQ)\ltimes \QQ^{n-1},$ given by the standard linear  action 
on $\QQ^{n-1}.$ 
\end{itemize}

\noindent 
(iii) Let  $\mathbf{G}\act X$ be a continuous  action of  a locally compact group $\mathbf{G}$ on a compact space $X.$
Let $G$ be a countable (not necessarily discrete) subgroup of $\mathbf{G}$ such that
the restricted action $G\act X$ is  a topologically free boundary action of   $G.$ 
Let $\mathbf{H}$ be an amenable closed subgroup of $\mathbf{G}$.
Then every subgroup $H$ of $G\cap \mathbf{H}$ belongs to $\Sub_{\rm w-par}(G)$.
Indeed, the locally compact amenable group $\mathbf{H}$ stabilizes a probability measure on $X$ and the claim follows.
Here are two specific examples.
\begin{itemize}
\item For $n\geq 2$, we claim that $H=PO_n(\QQ)$, the (projective) group of orthogonal matrices with rational entries, 
is a weakly parabolic subgroup of $G=PGL_n(\QQ).$
Indeed, let $\mathbf{G}=PGL_n(\RRR)$ 
 and  $\mathbf{G}\act X$ the standard action on $X=\PP(\RRR^n)$. 
The action $G\act X$  is a topologically free boundary action and
$H$ is contained in the compact subgroup $PO_n( \RRR)$ of   $\mathbf{G}.$ Observe that 
$H$ is not amenable when $n\geq 3,$ as it is dense in $PO_n( \RRR)$ and so contains non abelian free subgroups.
\item For $n\geq 2,$ we claim that  $H=PSL_n(\ZZ)$ is a weakly parabolic subgroup of $G=PSL_n(\QQ).$
Indeed, let $\mathbf{G}=PSL_n(\QQ_p)$ for  a prime integer $p$ 
and $\mathbf{G}\act X$ the standard action on $X=\PP(\QQ_p^n)$. 
 The action $G\act X$  is a topologically free boundary action and
$H$ is contained in the compact subgroup $PSL_n( \ZZ_p)$ of   $\mathbf{G}.$
\end{itemize}

\end{example}

\subsection{Some properties of weakly parabolic subgroups}
\label{SS:WPS-Properties}
The following property of weakly parabolic subgroups will be a crucial tool in the study of their associated  $C^*$-algebras
(see the proof of Theorem~\ref{Theo-C*IndRep}).
Recall that a map $\Phi: \A\to \B$ between two unital $C^*$-algebras is unital if $\Phi(\mathds{1}_\A)=\mathds{1}_\B$ and that 
$\Phi$ is positive if $\Phi(T)\geq 0$ for every $T\in \A$ with $T\geq 0.$
\begin{proposition}
\label{Prop-WPS-Boundary}
Let $G$ be a countable group and $H\in \Sub_{\rm w-par}(G).$
Then there exist a topologically free boundary action $G\act X$ 
 and a linear  isometric map $C(X) \to \ell^\infty(G/H)$
which is  $G$-equivariant, unital and positive.
\end{proposition}
\begin{proof}
Since $H\in \Sub_{\rm w-par}(G),$ there exists a topologically free boundary action $G\act X$ such that $X$ admits an $H$-invariant probability measure $\nu.$ Consider the Poisson transform $\P_\nu: C(X) \to \ell^\infty(G/H)$,  defined by 
$$
\P_\nu(f)(gH) = \int_Xf(gx) d\nu(x) \qquad\text{for} \qquad f\in C(X),\, g\in G.
$$
It is clear that $\P_\nu$ is  linear, $G$-equivariant, unital and positive.
So, $\P_\nu$ induces a $G$-equivariant map $\Phi:G/H\to \pr(X),$ given by 
$$\Phi(gH)(f)= \P_\nu(f)(gH) \qquad\text{for} \qquad f\in C(X), \ g\in G.$$
Since $G\act X$ is a boundary action, the range of $\Phi$ consists exactly of the point measures
$\delta_x$ for $x\in X$ (see \cite[Proposition 4.2]{Furstenberg}). 
So, there exists a surjective map $\varphi: G/H\to X$ such that $\Phi(gH)= \delta_{\varphi(gH)}$
for every $g\in G.$
It follows that $\P_\nu$ is isometric; indeed,  let $f\in C(X).$
Then 
$$
\begin{aligned}
\sup_{g\in G} \vert  \P_\nu(f)(gH)\vert&= \sup_{g\in G} \vert  \Phi(gH)(f)\vert\\
&= \sup_{g\in G} \vert  f(\varphi(gH))\vert\\
&= \sup_{x\in X} \vert  f(x)\vert.
\end{aligned}
$$
\end{proof}

The next result generalizes the 
well-known fact that a $C^*$-simple group contains  no  non-trivial amenable normal subgroup. Recall that, if $\Sub_{\rm w-par}(G)\neq\emptyset$, then $G$ is $C^*$-simple (Remark~\ref{Rem-Sub_g}).
\begin{proposition}
\label{Prop-Sub_g}
Let $G$  be a countable group. Then $\Sub_{\rm w-par}(G)$ contains no  non-trivial normal subgroup of $G.$
\end{proposition}

\begin{proof}
Let $N$ be a normal subgroup of $G$, and suppose there is a topologically free boundary action $G\act X$ and an $N$-invariant $\nu\in\pr(X)$. Since $N$ is normal, $g_*\nu$ is  $N$-invariant for every $g\in G$. 
As $G\act X$ is a boundary action, the point measure $\delta_x$ belongs to the the weak*-closure of $\{g_*\nu \mid g\in G\}$
for every $x\in X.$ Hence, $N$ acts trivially on $X$ and this implies that $N=\{e\}$, since  $G\act X$  is topologically free.
\end{proof}

Next, we establish an interesting property of the conjugacy class of weakly parabolic subgroups. 

\begin{theorem}
\label{thm:qus-reg-parab--w-contain-reg}
Let $G$ be a countable group. For $H\in\Sub_{\rm w-par}(G)$, the following hold:
\begin{itemize}
\item[(i)] $\{e\}\in \overline{\C(H)}$;
\item[(ii)] $\la_{G} \prec \la_{G/H}$.
\end{itemize}. 
\end{theorem}
\begin{proof}
Let $G\act X$ be a topologically free boundary action such that $H$ stabilizes a  probability measure $\nu$ on $X$. 
 Let $x\in X$ be a point with trivial stabilizer.
Then $\delta_x\in \pr(X)$ has a trivial stabilizer.  Since $G\act X$ is a boundary action, $\delta_x$ belongs to the weak*-closure of  $G\nu$. 
 Proposition~\ref{Cor1-Action}.ii applied  to the action $G\act \pr(X)$
 shows that Items (i) and (ii) hold, since  $H\subset G_\nu$. 
 \end{proof}

\begin{remark}
We will give below (Theorem~\ref{Theo-C*IndRep})  a much stronger version
of Theorem~\ref{thm:qus-reg-parab--w-contain-reg}.ii.
\end{remark}

A subgroup $H\in \Sub(G)$  is \emph{recurrent} if  $\{e\}$
does not belong to the closure of $\C(H)$ for the Chabauty topology;
this notion was introduced in \cite{Ken} and used there (see Theorem 1.1) to give the following  characterization of $C^*$-simplicity: a discrete group $G$ is $C^*$-simple if and only if $G$ has no amenable recurrent subgroup.
Since amenable subgroups are weakly parabolic (provided $\Sub_{\rm w-par}(G)\neq \emptyset$), the following corollary of Theorem~\ref{thm:qus-reg-parab--w-contain-reg} strengthens one  implication in  this result.
\begin{corollary}
\label{Cor-RecurrentSub}
Let $G$ be a countable group. Then $\Sub_{\rm w-par}(G)$ contains no  recurrent subgroup of $G$.
\end{corollary}

\section{$C^*$-algebras associated to induced representations}
\label{S: C*-algebras}
We are going to draw some consequences of the results from Section~\ref{S: RigidSubgroups} 
for the primitive ideal space of a countable group $G$; we also study the ideal theory of $C^*$-algebras associated to 
 induced representations of $G$ from  
 subgroups which either have  the spectral gap property or  are weakly parabolic.

\subsection{Primitive ideal space of $G$}
\label{SS:PrimG}
Recall that the \textbf{primitive ideal space} $\Prim (G)$ of $G$ is the set 
of equivalence classes of irreducible unitary representations of $G$ for the relation of weak equivalence.
  
Let $C^*(G)$ be the maximal  $C^*$-algebra of $G$ (see Subsection~\ref{SS-C*-W*-alg}). 
One may describe $\Prim(G)$ as the set $\{ C^*\ker(\pi) \mid \pi \in \widehat{G}\},$
where $C^*\ker(\pi)$ is the kernel of the extension of the representation $\pi$
to $C^*(G)$; for all this, see \cite[Chap.13]{Dix--C*}.

Let 
$$
X:=\left\{(H, \sigma)\mid H\in \Sub_{\rm sc}(G), \sigma \in \widehat{H}_{fd}\right\},
$$
where $\Sub_{\rm sc}(G)$ is the set of self-commensurating subgroups of $G.$
Recall from Subsection~\ref{S: IrredQuasiRegRep} that 
$G$ acts naturally on $X$ and that the map 
$$\Ind: X\to \Rep(G), \qquad (H,\sigma) \mapsto  \Ind_H^G \sigma$$
factorizes to an injective map 
$ X/G\to \widehat{G},$
 where $X/G$ is the space of orbits for 
that the natural $G$-action on $X$. This yields the map
$$X\to \Prim(G), \qquad (H,\sigma) \mapsto  C^*\ker(\Ind_H^G \sigma)$$
which factorizes through $X/G.$

Items (ii) and (iii) of Theorem~\ref{Theo-RepRigidSubgroups} may be rephrased in terms of 
the primitive ideal space of $G$ in the following way.

\begin{theorem}
\label{Theo-RepRigidSubgroups-PrimIdeal}
Let $G$ be a countable group. The  restriction of the map 
$$ X\to \Prim(G), \qquad (H,\sigma) \mapsto  C^*\ker(\Ind_H^G \sigma)$$ to the $G$-invariant subset
$$
X_{\rm sg}:=\left\{(H, \sigma)\mid H\in \Sub_{\rm sg}(G), \sigma \in \widehat{H}_{\rm fd}\right\}
$$
factorizes to an \emph{injective} map $X_{\rm sg}/G \to \Prim(G).$ 

Moreover, if $C^*\ker(\pi)=C^*\ker(\Ind_H^G \sigma)$ for some $\pi\in \widehat{G}$ and  $(H,\sigma)\in X_{\rm sg},$ then $\pi$ is  equivalent to $\Ind_H^G \sigma.$
\end{theorem}

\subsection{$C^*$-algebras associated to induced representations}
\label{SS:C*-algSpectGapSubgroup}
We prove a result about the ideal structure of  the $C^*$-algebra generated by 
an  induced representation from a weakly parabolic subgroup.

In the sequel, by an ideal in a $C^*$-algebra $\A,$ we mean a closed two-sided ideal of $\A.$

\begin{theorem}
\label{Theo-C*IndRep}
Let $G$ be a countable group, $H\in \Sub_{\rm w-par}(G)$ and $\sigma\in \Rep(H).$ 
Let $\pi:= \Ind_{H}^G\sigma$.
The $C^*$-algebra $C^*_{\pi}(G)$ contains a (unique)  largest  proper ideal $I_{\rm max}$
that is, a proper  ideal of $C^*_{\pi}(G)$ which contains every proper 
 ideal of  $C^*_{\pi}(G)$. Moreover, the map 
$$\pi(G)\to \la_G(G),  \qquad \pi(g)\mapsto \la_G(g)$$ extends
to a surjective $*$-homomorphism $C^*_{\pi}(G)\to \CRed(G)$ with kernel $I_{\rm max}.$
\end{theorem}
\begin{proof}
Let $\K$ be the Hilbert space of $\sigma.$ Recall that the Hilbert space $\H$ of $\pi= \Ind_H^G \sigma$ is
the space of maps $F: G\to \K$ such that 
$F(xh) = \sigma(h^{-1}) F(x)$ for all $x \in G, h \in H$
and such that $xH\mapsto \Vert F(x)\Vert$ belongs to $\ell^2(G/H).$
Every $f\in \ell^\infty(G/H)$ defines a bounded operator $M(f)$ on $\H,$
given by multiplication by $f:$ 
$$
(M(f) F)(x)= f(xH) F(x)  \qquad \text{for all} \qquad x\in G, \ F\in \H;
$$
 moreover, we have the covariance relation
 $$
 \pi(g) M(f)\pi(g^{-1}) = M({}_gf) \qquad \text{for all} \qquad g\in G , \leqno{(*)}
 $$
 where ${}_g f \in \ell^\infty(G/H)$ is defined by ${}_g f (xH)= f(g^{-1}xH).$

Since $H\in \Sub_{\rm w-par}(G),$ there  exist a topologically free boundary action $G\act X$ 
and a linear isometric map 
$$P:C(X) \to \ell^\infty(G/H)$$
which is $G$-equivariant, unital and positive (see Proposition~\ref{Prop-WPS-Boundary}).
Observe that $P$ is not  an algebra homomorphism in general. 
We will view $C(X)$ as a closed $G$-invariant subspace of  $\ell^\infty(G/H).$

Let $S(C^*_{\pi}(G))$ and $S(\B(\H))$  be the state spaces of $C^*_{\pi}(G)$ and $\B(\H)$ (see Subsection~\ref{SS-States }).
The action $G\act S(C^*_{\pi}(G))$ extends to an action $G\act S(\B(\H))$, given by the same formula: 
$$g\cdot \varphi(T)= \varphi (\pi(g) T \pi(g^{-1})) \quad\text{for} \quad g\in G, \  \varphi\in  S(\B(\H)), \ T\in  \B(\H).$$
 
Recall that the embedding $P: C(X) \to \ell^{\infty}(G/H)$ preserves positivity. 
Hence, for $\varphi\in  S(\B(\H))$, the map 
$$C(X) \to\CCC, \qquad f\mapsto \varphi(M(f)),$$
denoted by $\varphi|_{C(X)},$ is a state on 
$C(X)$ and is therefore given by a probability measure on $X.$ So,  $\{\varphi|_{C(X)}\mid \varphi\in  S(\B(\H))\}$
 can be identified with $\pr(X)$; moreover, since $C(X)$ is $G$-equivariantly embedded
 in $\ell^\infty(G/H)$ and in view of the relation $(*)$,  the restriction of the action $G\act S(\B(\H))$ to $\pr(X)$ under this identification corresponds to  the natural action of  $G$ on $\pr(X).$

\vskip.2cm
From now on, we  imitate some of the arguments given in the proof of \cite[Theorem 4.5]{Haagerup}, adapting them to our context.

\noindent
$\bullet$ \textbf{Step 1.} Let $\varphi \in S(C^*_{\pi}(G))$ and $x\in X$. 
Let $\varphi'\in S(\A)$ be an extension of $\varphi$. There exists a net
$(g_i)_{i}$ in $G$ such that $(g_i\cdot \varphi')_i$  converges to a state $\psi\in S(\A)$ with  $\psi|_{C(X)}=\delta_x.$

Indeed, let $\mu= \varphi'|_{C(X)}\in \pr(X).$ Since $G\act X$ is a boundary action, there exists
a net $(g_i)_{i}$ in $G$ such that $\lim_i g_i\cdot \mu= \delta_x.$
By compactness of $S(\A),$ by passing to a subnet if necessary, we can assume that $(g_i\cdot \varphi')_i$ converges to some $\psi\in S(\A)$.
Then $\lim_i (g_i\cdot \varphi')|_{C(X)}= \psi|_{C(X)}$ and, since
$(g_i\cdot \varphi')|_{C(X)}= g_i\cdot \mu$, we have $\psi|_{C(X)}=\delta_x.$

 \noindent
$\bullet$ \textbf{Step 2.} Let $\varphi \in S(C^*_{\pi} (G))$ and $x\in X$. 
Assume that $gx\neq x$ for every $g\in G\setminus \{e\}.$ 
There exists a state $\psi\in S(C^*_{\pi} (G))$ which belongs to the weak* closure of $\{g \cdot \varphi\mid  g\in G\}$
such that
$$
\psi(\pi(g))= 0  \qquad\text{for all} \quad g\in G\setminus\{e\}.
$$

Indeed, 
let  $\varphi'\in S(\A), g_i \in G$ and $\psi\in S(\A)$ be as in 
 Step 1. Let $g\in G\setminus\{e\}$. Since $gx\neq x$, we can choose $f\in C(X)$ such that 
$$0\leq f\leq \mathds{1}_{C(X)}, \quad f(g x) = 0, \quad \text{and} \quad f(x)=1.$$
Hence, we have $0\leq 1-f\leq \mathds{1}_{C(X)}$ and so
$$0\leq  M(\mathds{1}_{C(X)}-f) \leq M(\mathds{1}_{C(X)}).
$$
Therefore, 
$$
 \left(M(\mathds{1}_{C(X)}-f)\right)^2 \leq M(\mathds{1}_{C(X)}-f).
 $$
One the one hand, using the Cauchy-Schwartz inequality, it follows that
\[\begin{split}
\vert\psi\left((\mathds{1}_{\B(\H)}-M(f))\pi(g)\right)\vert^2 &=\vert\psi\left((M(\mathds{1}_{C(X)} -f))\pi(g)\right)\vert^2\\
 &\leq \psi\left(\left(M(\mathds{1}_{C(X)}-f)\right)^2\right)\psi(\mathds{1}_{\B(\H)}) \\&\leq \psi\left(M(\mathds{1}_{C(X)}-f)\right) \\
&= 1-\psi(M(f))\\
&= 1-f(x) = 0,
\end{split}\]
so that $\psi(\pi(g))= \psi (M(f)\pi(g)).$ 

 On the other hand, since $0\leq f\leq \mathds{1}_{C(X)},$ we have, 
 as above, 
 $$M(f)^2 \leq M(f).$$
 Using the fact that 
 $\psi(T^*)=\overline{\psi(T)}$ for any $T\in \B(\H)$ 
 and applying again the Cauchy-Schwartz inequality, it follows that
\begin{align*}
\vert\psi(M(f)\pi(g))\vert^2&=  \vert\psi\left(\pi(g^{-1})M(f))\right) \vert^2\\
&\leq \psi(\pi(g^{-1})M(f)^2\pi(g))\\
&\leq \psi(\pi(g^{-1})M(f)\pi(g))\\
&=\psi(M({}_{g^{-1}} f))\\
&= f(g x) = 0;
\end{align*}
so, $\psi(M(f)\pi(g))=0.$  Hence, $\psi(\pi(g)) = 0,$ as claimed.

Recall (see Subsection~\ref{SS-States }) that we can view $S(C^*_\rho(G))$ as a weak* closed subset of $S(C^*(G))$ for every $\rho\in \Rep(G)$.

\noindent
$\bullet$ \textbf{Step 3.} Let $\varphi\in S(C^*_\pi(G)).$ The canonical trace  $\tau_G$ 
 belongs to the weak*-closure of $\{g\cdot \varphi\mid g\in G\}$.

Indeed, since $G\act X$ is topologically free, we can find $x\in X$ with $gx\neq x$ for every  $g\in G\setminus \{e\}$.
Let $\psi\in S(C^*_{\pi} (G))$ be as in Step 2.
Then $\psi(\pi(g))= 0$ for every  $g\in G\setminus \{e\}$. Hence, 
$\psi= \tau_G,$ when $\psi$ and $\tau_G$ are viewed as elements of $S(C^*(G))$.  This proves the claim.

\noindent
$\bullet$ \textbf{Step 4.} Let $\rho\in \Rep(G)$ be such that $\rho\prec \pi.$ Then $\la_G\prec \rho.$

Indeed, let $\varphi\in S(C^*_\rho(G)).$ Since $\rho\prec \pi,$ we can view $\varphi$ as a state on $C^*_{\pi} (G).$
Hence, $\tau_G$ belongs to the weak*-closure of $\{g\cdot \varphi\mid g\in G\}$, by Step 3.
As $g\cdot \varphi\in S(C^*_\rho(G)),$ it follows that $\tau_G\in  S(C^*_\rho(G))$; hence,
$\la_G\prec \rho$, since $\la_G$ is cyclic.

\vskip.2cm 
In particular, it follows from Step 4 that $\la_G\prec \pi$ and so the map  
$\pi(g)\mapsto\la_G(g)$, extends to a surjective $*$-homomorphism $ C^*_{\pi} (G)\to \CRed (G)$. Let $I$ be the kernel of this $*$-homomorphism. So, $C^*_{\pi} (G)/I$ is isomorphic to $\CRed(G).$

\noindent
$\bullet$ \textbf{Step 5.}  We claim that  $I$ is a proper ideal of $C^*_{\pi} (G)$ that contains all other proper ideals of $C^*_{\pi} (G)$.

Indeed, observe first that $I$ is proper since it is the kernel of a non-zero homomorphism. 
Now, let $J$ be a proper ideal of $C^*_{\pi} (G)$. Then $J$ is the kernel of the extension to $C^*_{\pi} (G)$ of a
unitary representation $\rho$ of $G$. Then $\rho\prec \pi$ and hence, by Step 4, $\la_G\prec \rho.$
Therefore, $J\subset I.$
\end{proof}
\begin{remark}
\label{Rem-Theo-C*IndRep}
\begin{itemize}
\item[(i)]  Let $G$ be a countable group, $H\in \Sub_{\rm w-par}(G),$ and $\sigma\in \Rep(H)$.
Then, for  the ideal $I_{\rm max}$ as in Theorem~\ref{Theo-C*IndRep}, we have  
$$I_{\rm max}=\{0\}\Longleftrightarrow \sigma\prec \la_H.$$
  Indeed, since $\la_G\prec \Ind_H^G\sigma,$ we have (see  Corollary~\ref{Cor-AmenableInduced})
$$I_{\rm max}=\{0\}\Longleftrightarrow  \Ind_H^G \sigma \prec \la_G  \Longleftrightarrow\sigma\prec \la_H.$$
In particular, for $\sigma=1_H$ so that $\Ind_H^G\sigma=\la_{G/H},$ we have $I_{\rm max}=\{0\}$ if and only 
if $H$ is amenable.
\item[(ii)]  Let $G$ be a countable group such that there exists a topologically free boundary action $G\act X$.
Then  Theorem~\ref{Theo-C*IndRep} applied to the case $H=\{e\}\in\Sub_{\rm w-par}(G)$ shows that 
that $\CRed(G)$ is simple. Thus,  Theorem~\ref{Theo-C*IndRep} 
 generalizes the  main criterium for $C^*$-simplicity of $G$ from \cite{KalKen} (see Theorem~\ref{Theo-KK} above).
\item[(iii)]  Let $G$ be a countable group, $H\in \Sub_{\rm w-par}(G)$ and $\pi=\Ind_{H}^G$ for $\sigma\in \Rep(H).$ 
Using arguments as in the proof of \cite[Theorem 4.5]{Haagerup}, 
one can rephrase Theorem~\ref{Theo-C*IndRep} in terms of  a ``Dixmier approximation property'':
for every  finite subset  $F$ of $G\backslash\{e\}$ and $\ep>0$, there exist elements $g_1,  \dots, g_n \in G$ such that
$$
\left\| \frac1n \sum_{i=1}^n \pi(g_igg_i^{-1}) \right\| < \ep \qquad\text{for every} \quad g\in F.
$$
\end{itemize}
\end{remark}

\vskip.3cm

Next, we examine a consequence of Theorem~\ref{Theo-RepRigidSubgroups} 
for the ideal theory of the $C^*$-algebra generated by an   induced representation from a subgroup 
with the spectral gap property.
\begin{theorem}
\label{Theo-IdealC*Subsg}
Let $G$ be a countable group, $H\in \Sub_{\rm sg}(G)$ and $\sigma\in \widehat{H}_{\rm fin}.$ 
Let $\pi:= \Ind_{H}^G\sigma$.
The $C^*$-algebra $C^*_{\pi}(G)$ contains a (unique) smallest  non-zero ideal $I_{\rm min}$, that is, 
a non-zero ideal such that every non-zero ideal   $C^*_{\pi}(G)$ contains $I_{\rm min}.$
\end{theorem}
\begin{proof}
Theorem~\ref{Theo-RepRigidSubgroups} shows that $\pi$ is a traceable irreducible representation; hence, $C^*_{\pi}(G)$ contains the ideal $I=\K(\H)$ of compact operators on the Hilbert space $\H$ of $\pi.$
Let $J$ be a non-zero ideal of $C^*_{\pi}(G)$ . Then  clearly $JI \neq \{0\}$ 
and hence  $J\cap I \neq \{0\}$. Since $I$ is a simple $C^*$-algebra, it follows that $I\subset J.$
\end{proof}

For subgroups $H$ belonging to both $\Sub_{\rm sg}(G)$ and $\Sub_{\rm w-par}(G)$, we have the following 
consequence of Theorem~\ref{Theo-C*IndRep} and Theorem~\ref{Theo-IdealC*Subsg}.
\begin{corollary}
\label{thm:top-f-bnd-act-main}
Let $G$ be a countable group, $H\in \Sub_{\rm w-par}(G)\cap \Sub_{\rm sg}(G),$  and $\sigma\in \widehat{H}_{\rm fin}.$ 
Let $\pi:= \Ind_{H}^G\sigma$.
\begin{itemize}
\item[(i)] The $C^*$-algebra $C^*_{\pi}(G)$ contains a  smallest non-zero ideal.
\item[(ii)] The regular representation $\la_G$ extends to a representation of  $C^*_{\pi}(G)$
and the kernel   of this extension is the   largest proper ideal
of $C^*_{\pi}(G)$.
\item[(iii)] There exists  $\pi_{\rm min}\in \Rep(G)$ with the following property:
for every  $\rho\in \Rep(G)$ with $\rho\prec \pi$ and $\rho \centernot \sim \pi$, we have 
$$\lambda_G\prec \rho \prec \pi_{\rm min}\prec \pi.$$
 \end{itemize}
\end{corollary}
\begin{proof}
Items (i) and (ii) are immediate consequences of  Theorem~\ref{Theo-C*IndRep} and Theorem~\ref{Theo-IdealC*Subsg}.

Let $\pi_{\rm min}$ be a representation of $C_\pi^*(G)$ with kernel 
 the smallest non-zero ideal $I_{\rm min}$; the restriction of $\pi_{\rm min}$ to $G$ 
   is a unitary representation, again denoted by $\pi_{\rm min}$.

Let $\rho\in \Rep(G)$ with $\rho\prec \pi$ and $\rho \centernot \sim \pi$. 
The kernel $J$  of the extension of $\rho$ to $C^*_{\pi} (G)$ is   a non-zero proper  ideal of  $C^*_{\pi} (G)$.
Hence,  $I_{\rm min} \subset J \subset I_{\rm max},$ by Items (i) and (ii),
with  $I_{\rm max}$ the kernel of the extension of $\la_G$ to $C^*_{\pi} (G)$; this means that 
$\la_G\prec \rho \prec \pi_{\rm min}.$
\end{proof}

\begin{remark}
\label{Rem-thm:top-f-bnd-act-main}
It is conceivable - but we know of no example - that, for some  groups $G$ and subgroups $H\in \Sub_{\rm w-par}(G)\cap\Sub_{\rm sg}(G)$, the ideals $I_{\rm min}$ and $I_{\rm max}$ of  $C^*_{\la_{G/H}}(G)$ coincide. If this happens, then 
 $C^*_{\la_{G/H}}(G)$ has a unique non-trivial  ideal.
\end{remark}

\begin{example}
\label{Exa-SLnZSLmZ}
Let $G=PSL_n(\ZZ)$ for $n\geq 3$. Let 
$H$ be a standard copy of   $PSL_{n-1}(\ZZ)$ inside $G$, that is,  the image in $G$
of, say, 
$$
\left[\begin{array}{cc}
1& 0 \\
0 & SL_{n-1}(\ZZ).
\end{array}
\right].
$$
We claim that $H\in \Sub_{\rm w-par}(G)\cap \Sub_{\rm sg} (G).$

Indeed, the fact that $H\in \Sub_{\rm w-par}(G)$ is shown  in Example~\ref{Exa-Sub_g},  
To show that $H\in \Sub_{\rm sg} (G),$ we have to consider the two cases:
$n=3$ and $n\geq 4.$ 
\begin{itemize}
\item $n=3:$    Example~\ref{Exa-A-NormalSubg} shows that $H$ is a-normal.
Since $H$ is non-amenable, we have $H\in \Sub_{\rm sg}(G)$, by Proposition \ref{Prop-ANormalSpectralGap}.
\item $n\geq 4:$  the subgroup $H$ has Kazhdan's propery (T) since $n-1\geq  3.$
Moreover, $[1:0:\dots:0]$ is the unique periodic point of $H$ in $X=\PP(\RRR^n).$ It follows that $H$ is strongly self-commensurating. Hence, $H\in \Sub_{\rm sg}(G)$, by Proposition~\ref{Prop-StronglySelfCommKazhdan}.
\end{itemize}

 With similar arguments, it can be  shown that  every standard copy   of   $PGL_{n-1}(\QQ)$ inside $G=PGL_n(\QQ)$ belongs to  $\Sub_{\rm w-par}(G)\cap \Sub_{\rm sg} (G)$ for $n\geq 3.$
 \end{example}

\end{document}